\documentclass[11pt,a4paper,notitlepage]{article}
\usepackage{amsthm}
\usepackage{amsfonts}
\usepackage{amsmath}
\usepackage{mathtools}

\usepackage{amssymb}

\usepackage{graphicx}

\usepackage{pdfpages}

\usepackage{geometry}           
\usepackage{cite}
\usepackage{indentfirst}
\usepackage{color}
\usepackage{hyperref}
\hypersetup{
   colorlinks,
   citecolor=black,
   filecolor=black,
   linkcolor=black,
   urlcolor=black
}

\usepackage{blindtext}
\usepackage[utf8]{inputenc}

\newtheorem{thm}{Theorem}[section]
\newtheorem*{thm*}{Theorem}
\newtheorem{lem}[thm]{Lemma}
\newtheorem{prop}[thm]{Proposition}
\newtheorem{cor}[thm]{Corollary}

\theoremstyle{definition}
\newtheorem{defn}[thm]{Definition}

\newtheorem{exmp}[thm]{Example}

\newtheorem{rmk}[thm]{Remark}

\newcommand{\sgn}{\text{sgn}}

\newcommand{\RN}[1]{%
 \textup{\uppercase\expandafter{\romannumeral#1}}%
}

\newcommand{\Addresses}{{
 \bigskip
 \textsc{Jeremiah Horrocks Institute, University of Lancashire, Preston, PR1 2HE, United Kingdom}\par\nopagebreak
 \textit{E-mail address}: \texttt{jwilson30@lancashire.ac.uk}}}
 
 \title{Surface cluster algebra expansion formulae via loop graphs}
\author{Jon Wilson}
\date{}

\usepackage{fancyhdr}

\fancypagestyle{plain}{% emulate the plain style
 \fancyhf{}% clear all fields
 \fancyfoot[C]{\thepage}%
}

\usepackage[titles]{tocloft}

\usepackage{cite}

\usepackage{xcolor}
\newcommand*\circled[1]{\kern-2.5em%
  \put(0,4){\color{white}\circle*{18}}\put(-0.5,4){\circle{12}}%
  \put(-3,0){\color{black}\small#1}~~}
\usepackage{enumitem}

\usepackage{float}

\usepackage{stackengine}

\begin{document}

\maketitle

\begin{abstract}

In 2011 Musiker, Schiffler and Williams obtained expansion formulae for cluster algebras from orientable surfaces \cite{musiker2011positivity}. For singly and doubly notched arcs these formulae required the notion of $\gamma$-symmetric perfect matchings and $\gamma$-compatible pairs of $\gamma$-symmetric perfect matchings, respectively. We simplify and unify these approaches by considering good matchings of \textit{loop} graphs.

\end{abstract}

\pagenumbering{arabic}

\vspace{5mm}

\tableofcontents

\section{Introduction}

Fomin and Zelevinsky's cluster algebras are a particular class of commutative algebras whose generators, \textit{cluster variables}, are obtained iteratively from an initial collection of (algebraically independent) variables. These cluster variables have some seemingly miraculous properties; even though they are defined recursively as rational functions, surprisingly, they are always Laurent polynomials in the initial variables \cite{fomin2002cluster}. In the context of cluster algebras from triangulated surfaces \cite{fomin2018cluster}, a setting in which cluster variables correspond to tagged arcs on the surfaces, Musiker, Schiffler and Williams significantly strengthened this result. Namely, they showed the coefficients of these Laurent polynomials are always non-negative \cite{musiker2011positivity} -- a particularly elusive conjecture at the time. Their beautiful result established a remarkable connection between monomials of the Laurent polynomial and perfect matchings of \textit{snake} graphs. For plain arcs this connection may be thought of as a perfect fit, as there is a bijection between the monomials and the collection of all perfect matchings of the associated snake graph (with respect to an initial ideal triangulation). However, for a general tagged arc this correspondence is somehow lost; one must consider much larger graphs and restrict attention to the so called `$\gamma$-symmetric' perfect matchings' or `$\gamma$-compatible pairs' of `$\gamma$-symmetric' perfect matchings.

In this paper, in an effort to find an optimal framework for all tagged arcs, we introduce the notion of \textit{loop} graphs. Roughly speaking this is the result of gluing the end(s) of a snake graph to an existing tile(s). The following theorem shows that \textit{good} matchings of loop graphs effectively describe the cluster variable expansion of any tagged arc, with respect to any tagged triangulation. \newline

\noindent \textbf{Main Theorem} (Theorem \ref{loopexpansion}). {\textit{Let $T$ be a tagged triangulation of a bordered surface $(S,M)$. \newline Let $\mathcal{A}$ be the associated cluster algebra with principal coefficients with respect to $\Sigma_T = (\mathbf{x}_T,\mathbf{y}_T,B_T)$. \newline Then for any}\footnote{In this paper, our proof of the Main Theorem relies on the work of Musiker, Schiffler and Williams \cite{musiker2011positivity}, therefore, if $\gamma$ is a doubly notched arc and $T$ contains only plain arcs, or $\gamma$ is a plain arc and $T$ contains only doubly notched arcs, then we must assume $(S,M)$ is not a twice punctured closed surface. For more information on this technical assumption, we refer the reader to [Theorem 4.20, \cite{musiker2011positivity}].
Note that by the pioneering work of Gross, Hacking, Keel, Kontsevich \cite{gross2018canonical}, or Lee, Schiffler \cite{lee2015positivity}, the Laurent expansion is known to have non-negative coefficients; the debate is to whether equation (\ref{main theorem}) holds. That said, the techniques developed in \cite{geiss2023bangle} can be used to verify this case too.} tagged arc $\gamma$ of $(S,M)$, the Laurent expansion of $x_{\gamma} \in \mathcal{A}$ with respect to $\Sigma_T$ is given by:

\begin{equation}
\label{main theorem}
x_{\gamma} = \frac{1}{\text{cross}(\gamma, T)}\sum_{P} x(P)y(P),
\end{equation}

\noindent \textit{where the sum is taken over all good matchings $P$ of the loop graph $G_{\gamma,T}$.} \newline

In general, the reader may want to obtain the Laurent expansion of $x_{\gamma}$ with respect to an arbitrary choice of coefficients, so our assumption of principal coefficients in the Main Theorem may seem quite restrictive. However, by the `separation of additions' formula of Fomin and Zelevinsky \cite{fomin2007cluster}, it is enough to consider only principal coefficients. Indeed, if $z$ is a cluster variable in a cluster algebra (with any choice of coefficients), their result states $z$ is obtained by a certain specialisation (and rescaling) of the corresponding cluster variable in the cluster algebra with principal coefficients. \newline

We believe the Main Theorem will be particularly useful in obtaining skein relations between (generalised) tagged arcs and closed curves on $(S,M)$\footnote{Since the first pre-print of this paper was released, skein relations have been obtained by two parties. Their independent approaches followed Section \ref{latticesection} in different directions: Tsironis's proof employed a representation theoretic approach \cite{tsironis2025skein}, building on Remark \ref{repremark}; whereas Banaian, Kang, Kelley opted to generalise \textit{snake graph calculus} to loop graphs \cite{banaian2024skein} by harvesting the poset structure outlined in Lemma \ref{posetlemma} and Theorem \ref{loop order lattice}.}. As a direct consequence one could then use this to obtain bases for all (full rank) cluster algebras arising from surfaces [Appendix A, \cite{musiker2013bases}]. \newline \indent

\indent The paper is organised as follows. Section 2 gives a very brief overview of cluster algebras and their relation to triangulated surfaces. In Section 3 we recall the construction of snake graphs, and introduce a like-minded generalisation called \textit{loop graphs}. In Section 4 we then show how one can associate these loop graphs to tagged arcs, with respect to ideal triangulations. The main result of the paper is given in Section 5, which shows that one can compute cluster variable expansions via good matchings of loop graphs, for any surface cluster algebra. The proof of this result is given in Section 6, which revolves around reinterpreting statements about `$\gamma$-symmetric' perfect matchings and `$\gamma$-compatible pairs' of `$\gamma$-symmetric' perfect matchings, as statements about good matchings of loop graphs. Finally, in Section 7 we show the collection of good matchings of any (surface) loop graph can be naturally endowed with a lattice structure.

\section*{\large \centering Acknowledgements}

The author is grateful for the generous support they received from Christof Geiss' CONACyT-239255 grant, Daniel Labardini-Fragoso's grants: CONACyT-238754 and C\'{a}tedra Marcos Moshinsky, and their postdoctoral fellowship at IMUNAM. The author also thanks Anna Felikson and Ezgi Kantarcı Oğuz for helpful comments which improved the readability of the text. Finally, I thank the anonymous referees for their careful reading and their many valuable comments.

\section{Preliminaries}
\label{preliminariessection}

\subsection{Cluster algebras}

This section provides a brief review of (skew-symmetric) cluster algebras of geometric type. Let $n \leq m$ be positive integers. Furthermore, let $\mathcal{F}$ be the field of rational functions in $m$ independent variables. Fix a collection $X_1,\ldots,X_n, x_{n+1},\ldots, x_m$ of algebraically independent variables in $\mathcal{F}$. We define the \textit{coefficient ring} to be $\mathbb{ZP}:=\mathbb{Z}[x_{n+1}\ldots x_m]$.

\begin{defn}
A \textit{\textbf{(labelled) seed}} consists of a pair, $(\mathbf{x},\mathbf{y}, B)$, where

\begin{itemize}[leftmargin=6.8mm]

\item $\mathbf{x} = (x_1,\ldots x_n)$ is a collection of variables in $\mathcal{F}$ which are algebraically independent over $\mathbb{ZP}$,

\item $\mathbf{y} = (y_1,\ldots y_n)$ where $y_k = \displaystyle \prod_{j=n+1}^{m} x_j^{b_{jk}}$ for some $b_{jk} \in \mathbb{Z}$,

\item $B = (b_{jk})_{j,k \in \{1,\ldots, n\}}$ is an $n \times n$ skew-symmetric integer matrix.
\end{itemize}

The variables in any seed are called \textit{\textbf{cluster variables}}. The variables $x_{n+1},\ldots,x_m$ are called \textit{\textbf{frozen variables}}. We refer to $\mathbf{y}$ as the \textbf{\textit{choice of coefficients}}.
Furthermore, following the literature, we define $\mathbf{\hat{y}} := (\hat{y}_1,\ldots \hat{y}_n)$ where $\hat{y}_k := y_k\displaystyle \prod_{j \in [1,n]} x_{j}^{b_{jk}}$ for each $k \in [1,n]$.

\end{defn}

\begin{defn}

Let $(\mathbf{x},\mathbf{y},B)$ be a seed and let $i \in \{1,\ldots,n\}$. 

We define a new seed $\mu_{i}(\mathbf{x},\mathbf{y},B) := (\mathbf{x}',\mathbf{y}',B')$, called the \textit{\textbf{mutation}} of $(\mathbf{x},\mathbf{y},B)$ at $i$ where:

\begin{itemize}[leftmargin=6.8mm]

\item $\mathbf{x}' = (x'_1,\ldots x'_n)$ is defined by 

$$x'_i = \frac{\displaystyle \prod_{b_{ki} >0} x_k^{b_{ki}} + \prod_{b_{ki} <0} x_k^{-b_{ki}}}{x_i}$$

and setting $x_j' = x_j$ when $j \neq i$;

\item $\mathbf{y}'$ and $B' = (b'_{jk})$ are defined by the following rule:

\[   
b'_{jk} = 
     \begin{cases}
       -b_{jk},& \text{if $j=i$ or $k=i$,}\\
       b_{jk} + \max(0,-b_{ji})b_{ik} + \max(0,b_{ik})b_{ji},& \text{otherwise.}\\
     \end{cases}
\]

\end{itemize}

\end{defn}

\begin{defn}
\label{cluster algebra}
Fix an \textit{initial seed} $ (\mathbf{x},\mathbf{y}, B)$. If we label the \textit{initial cluster variables} of $ \mathbf{x}$ from $1,\ldots, n $ then we may consider the labelled n-regular tree $\mathbb{T}_n $. Each vertex in $ \mathbb{T}_n $ has $ n$ incident edges labelled $1,\ldots,n $. Vertices of $ \mathbb{T}_n $ represent seeds and the edges correspond to mutation. In particular, the label of the edge indicates which direction the seed is being mutated in. 

Let $\mathcal{X} $ be the set of all cluster variables appearing in the seeds of $ \mathbb{T}_n$. The \textit{\textbf{cluster algebra}} of the seed $ (\mathbf{x},\mathbf{y}, B) $ is defined as $ \mathcal{A}(\mathbf{x},\mathbf{y}, B ) := \mathbb{ZP}[\mathcal{X}]$. \newline \indent We say $\mathcal{A}(\mathbf{x},\mathbf{y}, B)$ is the \textbf{\textit{cluster algebra with principal coefficients}} if $m = 2n$ and $\mathbf{y} = (y_1, \ldots, y_n)$ satisfies $y_k = x_{n+k}$ for any $k \in \{1,\ldots, n\}$.

\end{defn}

Cluster algebras with principal coefficients admit a natural $\mathbb{Z}^n$-grading which we describe in the following proposition.

\begin{prop}[Proposition 6.1, Corollary 6.2, \cite{fomin2007cluster}]

Let $\mathcal{A}(\mathbf{x},\mathbf{y}, B)$ be a cluster algebra with principal coefficients. Then every cluster variable of $\mathcal{A}(\mathbf{x},\mathbf{y}, B) \subseteq \mathbb{Z}[x_1^{\pm 1},\ldots, x_n^{\pm 1}, y_1, \ldots, y_n ]$ is homogeneous with respect to the $\mathbb{Z}^n$-grading determined by setting: $$\text{deg}(x_i) := e_i \hspace{7mm} \text{and} \hspace{7mm} \text{deg}(y_i) = \text{deg}(x_{n+i}) := -\sum_{k=1}^{n} b_{ki}e_k.$$

\noindent for each $i \in \{1,\ldots, n\}$, where $e_1,\ldots, e_n$ denote the standard basis vectors of $\mathbb{Z}^n$. \newline

Moreover, with respect to this grading, for each cluster variable $X \in \mathcal{A}(\mathbf{x},\mathbf{y}, B)$ we define the $\textbf{g-vector}$ of $X$ as: $$\textbf{g}(X) = (g_1,\ldots, g_n) := \text{deg}(X) \in \mathbb{Z}^n.$$

\end{prop}

\subsection{Cluster algebras from surfaces}
\label{Cluster algebras from surfaces}

In this subsection we recall the work of Fomin, Shapiro and Thurston \cite{fomin2008cluster}, which establishes a cluster structure for triangulated orientable surfaces. \newline \indent

Let $S$ be a compact orientable $2$-dimensional manifold. Fix a finite set $M$ of marked points of $S$ such that each boundary component contains at least one marked point -- we refer to marked points in the interior of $S$ as \textit{punctures}. The pair $(S,M)$ is called a \textit{\textbf{bordered surface}}. For technical reasons we exclude the cases where $(S,M)$ is an unpunctured or once-punctured monogon; a digon; a triangle; or a once, twice or thrice punctured sphere.

\begin{defn}

An \textit{\textbf{arc}} of $(S,M)$ is a simple curve in $S$ connecting two marked points of $M$, which is not isotopic to a boundary segment or a marked point.
\end{defn}

\begin{defn}

A \textit{\textbf{tagged arc}} $\gamma$ is an arc whose endpoints have been `tagged' in one of two ways; \textit{\textbf{plain}} or \textit{\textbf{notched}} (see Figure \ref{TaggedArcs}). Moreover, this tagging must satisfy the following conditions: if the endpoints of $\gamma$ share a common marked point, they must receive the same tagging; and an endpoint of $\gamma$ lying on the boundary $\partial S$ must always receive a plain tagging. In this paper we shall always consider tagged arcs up to isotopy.

\end{defn}

\begin{figure}[H]
\begin{center}
\includegraphics[width=10cm]{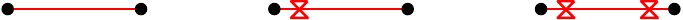}
\caption{Illustrations of a plain arc; a singly notched arc; and a doubly notched arc, respectively.}
\label{TaggedArcs}
\end{center}
\end{figure}

\begin{defn}

Let $\alpha$ and $\beta$ be two tagged arcs of $(S,M)$. We say $\alpha$ and $\beta$ are \textit{\textbf{compatible}} \textit{if and only if} the following conditions are satisfied:

\begin{itemize}
\item There exist isotopic representatives of $\alpha $ and $\beta$ that don't intersect in the interior of $S$. 

\item Suppose the untagged versions of $ \alpha$ and $\beta $ do not coincide. If $ \alpha$ and $\beta$ share an endpoint $p$ then the ends of $\alpha $ and $\beta $ at $p $ must be tagged in the same way.

\item Suppose the untagged versions of $\alpha $ and $\beta $ do coincide. Then precisely one end of $ \alpha $ must be tagged in the same way as the corresponding end of $\beta$.

\end{itemize}

A \textbf{\textit{tagged triangulation}} of $(S,M) $ is a maximal collection of pairwise compatible tagged arcs of $(S,M) $. Moreover, this collection is forbidden to contain any tagged arc that enclose a once-punctured monogon.

An \textit{\textbf{ideal triangulation}} of $ (S,M)$ is a maximal collection of pairwise compatible plain arcs. Note that ideal triangulations decompose $(S,M)$ into triangles, however, the sides of these triangles may not be distinct; two sides of the same triangle may be glued together, resulting in a \textit{\textbf{self-folded triangle}} -- the self-folded edge of such a triangle is called a \textit{\textbf{radius}} (see Figure \ref{Triangulations}).
\end{defn}

\newpage

\begin{rmk}

To each tagged triangulation $T$ we may uniquely assign an ideal triangulation $T^{\circ}$ as follows:

\begin{itemize}

\item  If $p$ is a puncture with more than one incident notch, then replace all these notches with plain taggings.

\item  If $ p$ is a puncture with precisely one incident notch, and this notch belongs to $ \beta \in T$, then replace $\beta $ with the unique arc $\gamma $ of $(S,M)$ which encloses $\beta $ and $p $ in a monogon.

\end{itemize}

Conversely, to each ideal triangulation $T $ we may uniquely assign a tagged triangulation $ \iota(T)$ by reversing the second procedure described above.
\end{rmk}

\begin{figure}[H]
\begin{center}
\includegraphics[width=13cm]{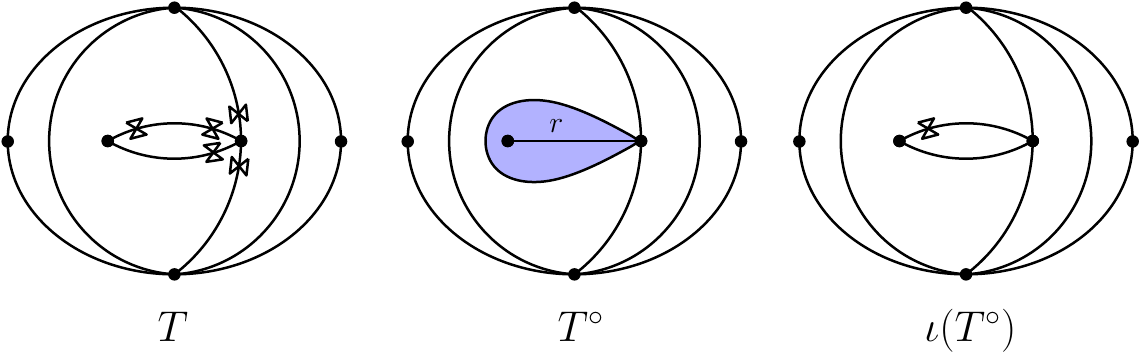}
\caption{An illustration of a tagged triangulation $T$, its corresponding ideal triangulation $T^{\circ}$, and the associated tagged triangulation $\iota(T^{\circ})$. The shaded triangle in $T^{\circ}$ is an example of a self-folded triangle with radius $r$.}
\label{Triangulations}
\end{center}
\end{figure}

\begin{defn}
\label{adjacency matrix}
Let $T$ be a tagged triangulation, and consider its associated ideal triangulation $ T^{\circ}$. We may label the arcs of $T^{\circ} $ from $ 1,\ldots, n$ (note this also induces a natural labelling of the arcs in $ T$). We define a function, $ \pi_{T} : \{1,\ldots,n\} \rightarrow \{1,\ldots,n\} $, on this labelling as follows:
\[   
\pi_{T}(i) = 
     \begin{cases}
       j & \text{if $ i $ is the glued side of a self-folded triangle in $ T^{\circ}$, and $j$ is the remaining side;}\\
       i & \text{otherwise.}\\
     \end{cases}
\]

For each non-self-folded triangle $ \Delta$ in $T ^{\circ} $, as an intermediary step, define the matrix $ B_T^{\Delta} = (b^{\Delta}_{jk}) $ by setting \[   
b^{\Delta}_{jk} = 
     \begin{cases}
       1 & \text{if $\Delta$ has sides $ \pi_{T}(j)$ and $ \pi_{T}(k)$, and $ \pi_{T}(k) $ follows $ \pi_{T}(j) $ in a clockwise order;}\\
       -1 & \text{if $ \Delta$ has sides $ \pi_{T}(j)$ and $\pi_{T}(k) $, and $ \pi_{T}(k) $ follows $\pi_{T}(j)$ in an  anti-clockwise order;}\\
       0 & \text{otherwise}\\
     \end{cases}
\]

The \textit{\textbf{adjacency matrix}} $ B_T = (b_{ij}) $ of $ T $ is then defined to be the following summation, taken over all non-self-folded triangles $ \Delta$ in $ T^{\circ}$:
$$B_T:=  \sum\limits_{\Delta} B_T^{\Delta} $$

\end{defn}

\begin{defn}

Let $T $ be a tagged triangulation of a bordered surface $ (S,M)$. Consider the initial seed $ (\mathbf{x},\mathbf{y},B_T) $, where: $ \mathbf{x} $ contains a cluster variable for each tagged arc in $T $; $B_T $ is the matrix defined in Definition \ref{adjacency matrix}; and $ \mathbf{y}$ is any choice of coefficients. We call $ \mathcal{A}(\mathbf{x}, \mathbf{y}, B_T)$ a \textbf{\textit{surface cluster algebra}}. 

\end{defn}

\begin{prop}[Theorem 7.9, \cite{fomin2008cluster}]

Let $T $ be a tagged triangulation of a bordered surface $(S,M)$. Then for any $\gamma \in T$ there exists a unique tagged arc $\gamma'$ of $(S,M)$ such that $\mu_{\gamma}(T) := T\setminus \{\gamma\} \cup \{\gamma'\}$ is a tagged triangulation. We call $\mu_{\gamma}(T)$ the \textbf{\textit{flip}} of $T$ with respect to $\gamma$.

\end{prop}

\begin{thm}[Theorem 6.1, \cite{fomin2018cluster}]
\label{surface correspondence}
Let $(S,M) $ be a bordered surface. If $ (S,M)$ is not a once punctured closed surface, then in the cluster algebra $\mathcal{A}(\mathbf{x}, \mathbf{y}, B_T)$, the following correspondence holds:
\begin{align*}
 &\hspace{23mm} \mathbf{\mathcal{A}(\mathbf{x}, \mathbf{y}, \text{$B_T$})} & &  &\mathbf{(S,M)} \hspace{26mm}&   \\ 
 &\hspace{18mm} \textit{Cluster variables}  &\longleftrightarrow&  &\textit{Tagged arcs} \hspace{23mm} &  \\
 &\hspace{25mm}\textit{Clusters}   &\longleftrightarrow&  &\textit{Tagged triangulations} \hspace{14mm}&  \\
 &\hspace{24mm} \textit{Mutation}    &\longleftrightarrow&   &\textit{Flips of tagged arcs} \hspace{16mm}&  \\
\end{align*}

When $(S,M)$ is a once-punctured closed surface then cluster variables are canonically in bijection with all plain arcs or all doubly notched arcs depending on whether $T$ consists solely of plain arcs or doubly notched arcs, respectively.
\end{thm}

\section{Snake and loop graphs}

In this section we first recall the notion of an (abstract) \textit{snake} graph, as seen in \cite{canakci2013snake}, \cite{canakci2015snake}. Following this, we define \textit{loop} graphs which are a generalisation of Musiker, Schiffler and Williams' band graphs \cite{musiker2013bases}.

\subsection{Snake graphs}

\begin{defn}

A \textbf{\textit{tile}} is a connected graph comprising of four vertices and four edges, where each vertex has degree two. \newline 
\indent We shall always embed this graph in the plane, viewing a tile as a square whose edges are parallel to the $ x $ and $ y $ axes. With respect to this embedding we label the edges North (N), East (E), South (S) and West (W), as shown in Figure \ref{tile}. \newline Throughout this paper, the \textit{\textbf{diagonal}} of a tile is the line connecting the North-West and South-East vertices. Note that the diagonal is not considered an edge of the tile.

\begin{figure}[H]
\begin{center}
\includegraphics[width=3cm]{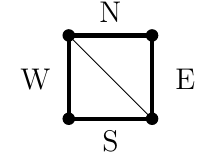}
\caption{A tile with North, East, South and West edge labellings. The associated diagonal is illustrated by the thin line.}
\label{tile}
\end{center}
\end{figure}

\end{defn}

Two tiles are said to be \textbf{\textit{glued}} if they share a common edge. The graphs considered in this paper will all be obtained via the process of gluing tiles, and the following definition gives the basic recipe of this.

\begin{defn}

A \textbf{\textit{snake graph}}, $ G = (G_1, \ldots, G_d) $, is a sequence of tiles $ G_1, \ldots, G_d $ such that the following holds for each $ i \in \{1,\ldots, d-1\} $:

\begin{itemize}

\item the North or East edge of $ G_i $ is glued to the South or West edge of $ G_{i+1 } $,

\item $ G_i $ and $ G_{i+1 } $ are glued at precisely one edge.

\end{itemize}

A subsequence of consecutive tiles $ G_i, G_{i+1}, \ldots, G_{j-1}, G_j $ occurring in $ G $ is called a \textbf{\textit{sub (snake) graph}} of $G $.

\end{defn}

\begin{defn}

Let $ G = (G_1,  \ldots , G_d) $ be a snake graph.

\begin{itemize} 

\item If the North (resp. East) edge of $ G_i $ is glued to the South (resp. West) edge of $ G_{i+1} $ for every $ i \in \{1,\ldots, n-1\} $, then $ G $ is said to be \textbf{\textit{straight}}.

\item $ G$ is said to be \textbf{\textit{zig-zag}} if no three consecutive tiles $ G_i, G_{i+1}, G_{i+2} $ of $ G$ form a straight sub snake graph.

\end{itemize}

\end{defn}

\begin{figure}[H]
\begin{center}
\includegraphics[width=9cm]{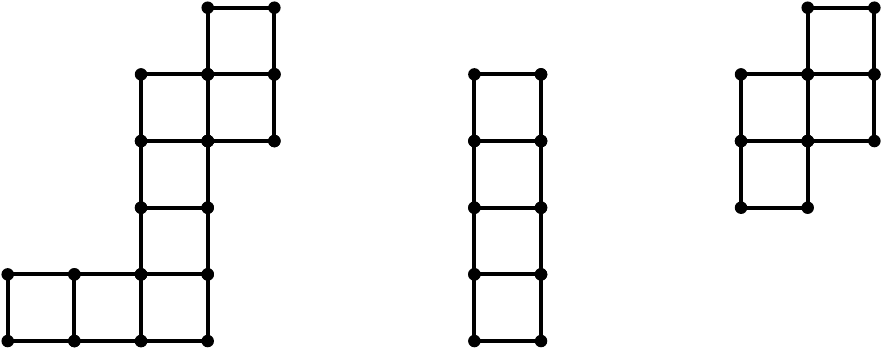}
\caption{On the left we show a general snake graph $G$. The two snake graphs on the right are straight and zig-zag sub snake graphs of $G$, respectively.}
\label{Fig_snakegraph}
\end{center}
\end{figure}

\begin{defn}

We define a \textbf{\textit{sign function}} on a snake graph $G$ to be a function, \sgn, from the set of edges of $G$ to $\{+,-\}$ such that, for each tile $G_i$ of $G$, we have: $\sgn(N(G_i)) = \sgn(W(G_i))$, $\sgn(S(G_i)) = \sgn(E(G_i))$, and $\sgn(N(G_i))  \neq \sgn(S(G_i)).$ See Figure \ref{snakegraphshade}.

\end{defn}

\begin{figure}[H]
\begin{center}
\includegraphics[width=6cm]{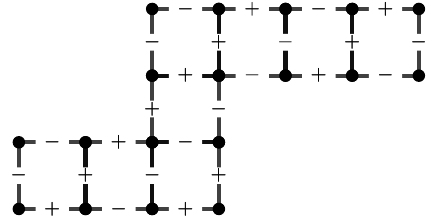}
\caption{An example of a sgn function on a snake graph. Interchanging `$+$' and `$-$' yields the other valid sgn function on the snake graph.}
\label{snakegraphshade}
\end{center}
\end{figure}

\begin{defn}
 
A \textit{\textbf{perfect matching}} of a graph $ G $ is a collection of edges of $ G $ such that every vertex of $ G $ is contained in exactly one of these edges.
 
\end{defn}

\begin{prop}[Section 4.3, \cite{musiker2010cluster}]
\label{induceddef}

Any perfect matching $P$ of a snake graph $G = (G_1,\ldots, G_d)$ induces an orientation on the diagonal of each tile $G_i$.\newline \indent Specifically, this is induced by travelling from the south-west vertex of $G_1$ to the north-east vertex of $G_d$ by alternating travel along edges in $P$ and diagonals of $G$. Following this procedure, each diagonal of $G$ (and each edge in $P$) will be traversed precisely once.

\end{prop}

\begin{figure}[H]
\begin{center}
\includegraphics[width=7cm]{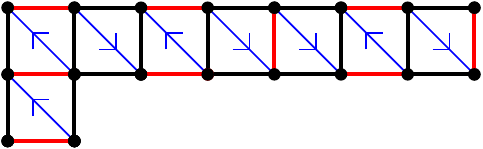}
\caption{A perfect matching $P$ of a snake graph. The orientation $P$ induces on the diagonals is shown in blue.}
\label{induceddiagonal}
\end{center}
\end{figure}

\subsection{Loop graphs}

Roughly speaking, a loop graph is obtained from a snake graph by (potentially) gluing tiles together. Just as snake graphs enable us to obtain expansion formulae for cluster variables corresponding to plain arcs \cite{musiker2011positivity}, loop graphs will provide us with the framework to write expansion formulae for the variables corresponding to any tagged arc.

\begin{defn}
\label{loopgraphrules}

Let $G = (G_1, \ldots, G_d)$ be a snake graph.

Consider $j,k \in \{1,\ldots, d\}$ where $j<k$ and let $c \in \{S(G_j),W(G_j)\}$ and $c' \in \{S(G_k),W(G_k)\}$ be boundary edges of $G$ such that $sgn(c) \neq sgn(c')$. Furthermore, let $x \in c$ denote the South-West vertex of $G_j$, and let $y$ denote the remaining vertex of $c$. Likewise, let $y' \in c'$ denote the South-West vertex of $G_k$, and let $x'$ denote the remaining vertex of $c'$. The associated \textbf{\textit{loop with respect to $G_j$ and $G_k$}} is obtained by gluing $c$ to $c'$ such that $x$ (resp. $y$) is glued to $x'$ (resp. $y'$).

We define a \textbf{\textit{loop with respect to $G_j$ and $G_k$}} where $j > k$ analogously; one should just replace South and West with North and East, respectively.

Finally, for $j,k \in \{1,\ldots, d\}$ where $j<k$, let $c \in \{S(G_j),W(G_j)\}$ and $c' \in \{N(G_k),E(G_k)\}$ be boundary edges of $G$ such that $sgn(c) = sgn(c')$. Furthermore, let $x \in c$ denote the South-West vertex of $G_j$, and let $y$ denote the remaining vertex of $c$. Likewise, let $y' \in c'$ denote the North-East vertex of $G_k$, and let $x'$ denote the remaining vertex of $c'$. The associated \textbf{\textit{band with respect to $G_j$ and $G_k$}} is obtained by gluing $c$ to $c'$ such that $x$ (resp. $y$) is glued to $x'$ (resp. $y'$).

In all scenarios above (see Figure \ref{loopbandgluing} for examples of all three), by abuse of notation, the resulting edge obtained from gluing $c$ and $c'$ is also denoted by $c$, and we call this the \textbf{\textit{cut}}.

\end{defn}

\begin{rmk}
Note that for a snake graph $G = (G_1, \ldots, G_d)$, if one considers the band with respect to tiles $G_1$ and $G_d$ then one recovers the so-called \textit{band graphs} of Musiker, Schiffler and Williams associated to (two-sided) closed curves \cite{musiker2013bases}. Moreover, if one replaces the gluing condition of $sgn(c) = sgn(c')$ with $sgn(c) \neq sgn(c')$ then one recovers the definition of band graph appearing in \cite{wilson2019positivity} associated to one-sided closed curves.
\end{rmk} \vspace{-5mm}

\begin{figure}[H]
\begin{center}
\includegraphics[width=16.5cm]{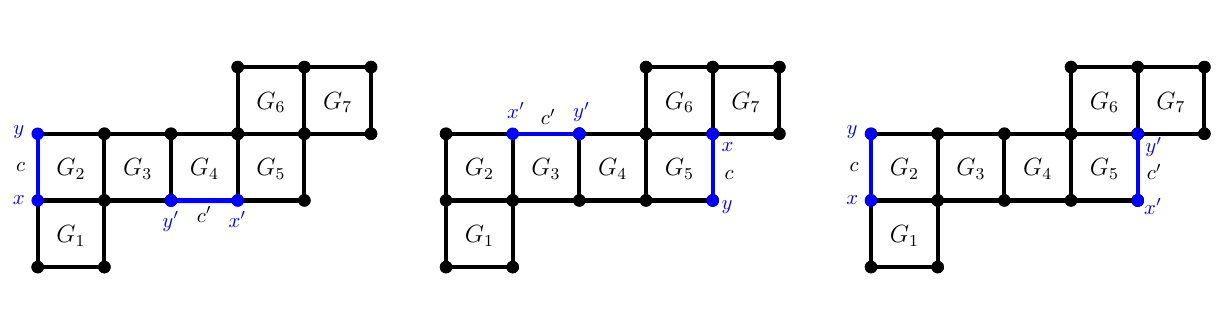}\vspace{-7mm}
\caption{We illustrate all three types of gluing: a loop with respect to tiles $G_2$ and $G_4$; a loop with respect to tiles $G_5$ and $G_3$; and a band with respect to tiles $G_2$ and $G_5$}
\label{loopbandgluing}
\end{center}
\end{figure} \vspace{-5mm}

\begin{defn}
\label{loopgraph}
A \textbf{\textit{loop graph}} $G^{\bowtie}$ is obtained from a snake graph $G = (G_1, \ldots, G_d)$ by creating a loop with respect to $G_1$ and $G_{k_1}$ where $k_1 \in \{2,\ldots, d\}$ and creating a loop with respect to $G_d$ and $G_{k_2}$ where $k_2 \in \{1,\ldots, d-1\}$ (see Figure \ref{loopgluingglued}). We extend the definition to all $k_1,k_2 \in \{1,\ldots, d\}$ by demanding there is no loop with respect to $G_1$ (resp. $G_d$) if $k_1 = 1$ (resp. $k_2 = d$). \newline
Furthermore, we allow the possibility of a band between tiles $G_2$ and $G_{d-1}$. \newline

\end{defn}

In summary, loop graphs are snake graphs for which zero, one, or two ends have been glued, and where there \textit{may} be a band between the second and penultimate tile.

\begin{figure}[H]
\begin{center}\vspace{-2mm}
\includegraphics[width=16.5cm]{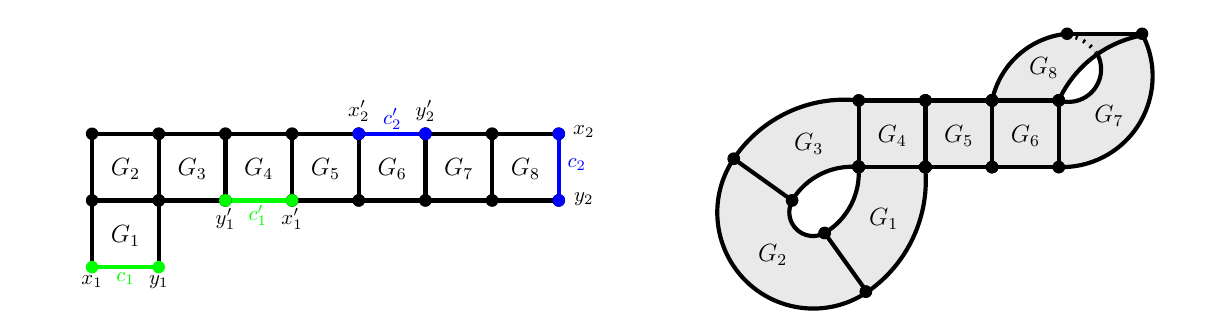}\vspace{-2mm}
\caption{Here we illustrate two visualisations of the same loop graph (which has loops at both of its ends). Throughout this paper we shall depict loop graphs as on the left.}
\label{loopgluingglued}
\end{center}
\end{figure}

\begin{defn}
\label{good matching}
 
As in Definition \ref{loopgraph}, let $G^{\bowtie}$ be a loop graph obtained from a snake graph $G$. A perfect matching $P$ of $G^{\bowtie}$ is called a \textit{good matching} if: \begin{itemize}

\item $P$ extends to a perfect matching $\overline{P}$ of $G$;

\item and for each cut $c$ of $G^{\bowtie}$ we have that $c \in \overline{P}$ or $c' \in \overline{P}$.

\end{itemize}

\end{defn}

\newpage

\begin{defn} By Definition \ref{good matching}, if $P$ is a good matching of a loop graph $G^{\bowtie}$, then it can be (uniquely) extended to a perfect matching $\overline{P}$ of the underlying snake graph $G$ (see Figure \ref{badmatching} for an example of a perfect matching of a loop graph that does not extend to a perfect matching of the underlying snake graph). For each cut $c$ of $G^{\bowtie}$ there are three possible scenarios (with respect to the good matching $P$):

\begin{itemize}

\item  $c \notin \overline{P}$ and $c' \in \overline{P}$, and we say $P$ is a \textbf{\textit{right cut}} with respect to $c$;

\item $c \in \overline{P}$ and $c' \notin \overline{P}$, and we say $P$ is a \textbf{\textit{left cut}} with respect to $c$;

\item $c, c' \in \overline{P}$, and we say $P$ is a \textbf{\textit{centre cut}} with respect to $c$.

\end{itemize}

\end{defn}

\begin{figure}[H]
\begin{center}\vspace{-3mm}
\includegraphics[width=8cm]{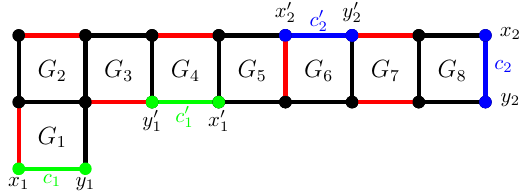}
\caption{An example of a perfect matching $P$ (indicated by the red edges) of a loop graph which is \textit{not} a good matching. Indeed, whilst $P$ is a left cut with respect to $c_2$, it is not a right, left nor centre cut with respect to $c_1$.}
\label{badmatching}
\end{center}
\end{figure}

\begin{prop}\vspace{-4mm}

As in Proposition \ref{induceddef}, any good matching $P$ of a loop graph $G^{\bowtie}$ induces an orientation on the diagonal of each tile of $G^{\bowtie}$.

\end{prop}

\begin{proof}

By definition of a good matching, $P$ can be (uniquely) extended to a perfect matching of the underlying snake graph $G$. The result then follows from Proposition \ref{induceddef}.

\end{proof} \vspace{-10mm}

\section{Snake and loop graphs from surfaces}
\label{snakebandconstruction}

Following the work of Musiker, Schiffler and Williams we first explain how to associate snake graphs to plain arcs of $(S,M)$ with respect to ideal triangulations \cite{musiker2011positivity}. We then generalise this approach and associate loop graphs to all tagged arcs. The basic principle is that tiles in our snake and loop graphs correspond to quadrilaterals on the surface.

\subsection{Snake graphs associated to plain arcs}

\subsubsection{T is an ideal triangulation without self-folded triangles}

Throughout this section we let $T = \{\tau_1, \ldots, \tau_n\}$ be a fixed ideal triangulation containing no self-folded triangles.

\begin{defn}
\label{relativeorientation}

Let $\gamma$ be a directed plain arc in $(S,M)$, and denote by $p_1, \ldots, p_d$ the intersection points of $\gamma$ with the arcs of $T$ (listed in order of intersection). In this way we obtain a sequence $i_1, \ldots, i_d$ such that $p_k$ belongs to $\tau_{i_k}$ for each $k \in \{1,\ldots, d\}$. \newline \indent
Let $Q_{i_j}$ be a quadrilateral in $T$ with diagonal labelled by $\tau_{i_j}$. Let $\Delta_j$ and $\Delta_{j+1}$ denote the triangles in $T$ either side of $\tau_{i_j}$, labelled so that, with respect to the orientation of $ \gamma $ through $ p_j $, $ \Delta_j $ precedes $ \Delta_{j+1} $. We view $Q_{i_j}$ as a tile $G_j$ by deleting the diagonal $\tau_{i_j}$ and embedding it in the plane so that:

\begin{itemize}

\item the (deleted) diagonal $\tau_{i_j}$ of $Q_{i_j}$ connected the north-west and south-east vertices of $G_j$.

\item $\Delta_j$ forms the lower half of $G_j$.

\end{itemize}

There are two possible ways to follow the rules above: if the orientation on $ Q_{i_j} $ (induced by $ (S,M) $) agrees with the orientation on $ G_j $ (induced by the clockwise orientation of the plane) then we write $ rel(G_j) =1 $; if it disagrees then $rel(G_j) = -1$. The quantity $rel(G_j)$ is called the \textbf{\textit{relative orientation}} of the tile $G_j$ with respect to $(S,M)$. See Figure \ref{tileorientation}.

\end{defn}

\begin{figure}[H]
\begin{center}
\includegraphics[width=13cm]{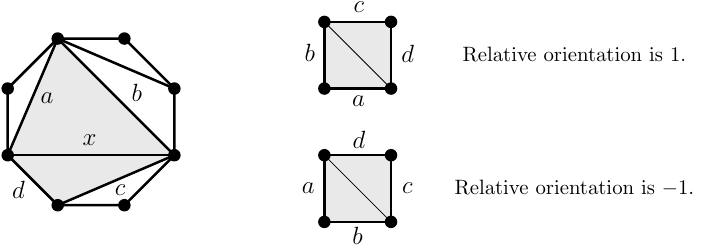}
\caption{An example of relative orientation. The thin line indicates the deleted diagonal.}
\label{tileorientation}
\end{center}
\end{figure}

\begin{defn}
\label{arcsnakegraph}

Let $\gamma$ be a plain arc in $(S,M)$. As in Definition \ref{relativeorientation} we associate a tile $G_i$ for each intersection point, $p_i$, such that $rel(G_i) \neq rel(G_{i+1})$. For each $j \in \{1,\ldots, d-1\}$ note that $\tau_{i_j} $ and $\tau_{i_{j+1}} $ form two sides of the triangle $\Delta_j$; we denote the remaining side by $\tau_{[i_j]}$. The snake graph $G_{\gamma,T} = (G_1,\ldots, G_d)$ associated to $\gamma$ and $T$ is determined by gluing $G_j$ to $G_{j+1}$ along their common edge $\tau_{[i_j]}$. See Figure \ref{surfacesnakenew}.

\end{defn}

\begin{figure}[H]
\begin{center}\vspace{-2mm}
\includegraphics[width=15cm]{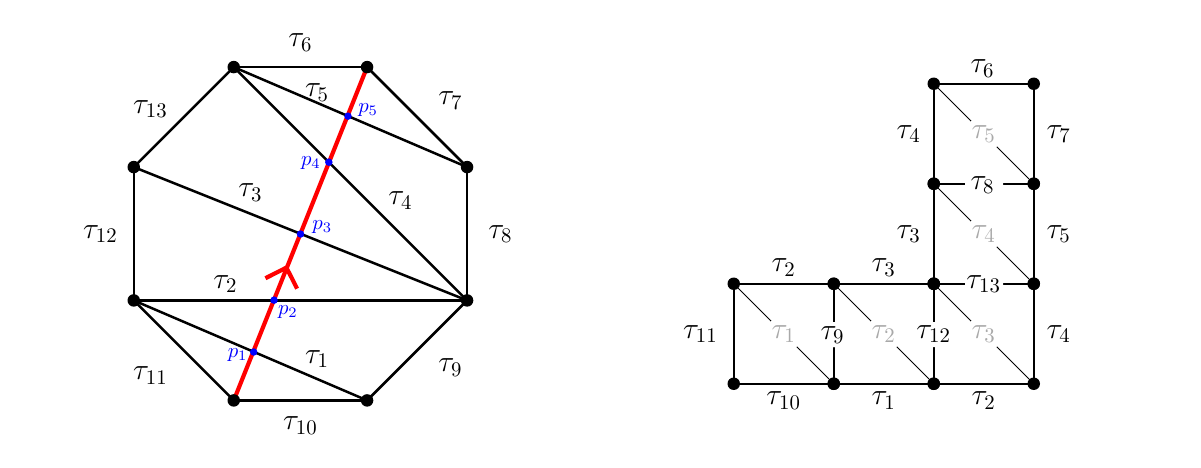}
\caption{The construction of a snake graph from a plain arc. The thin lines indicate the associated diagonals of the tiles, however, these are not considered edges of the actual snake graph.}
\label{surfacesnakenew}
\end{center}
\end{figure}

\subsubsection{T is an arbitrary ideal triangulation}

We now consider the general setup when $T = \{\tau_1, \ldots, \tau_n\}$ is an arbitrary ideal triangulation. Note that in this situation, ambiguities may arise if we followed the gluing procedures outlined in Definition \ref{arcsnakegraph}. Specifically, ambiguity occurs precisely when passing through a self-folded triangle. Musiker, Schiffler and Williams \cite{musiker2011positivity} addressed these points using the following two definitions.

\begin{defn}
Let $\gamma$ be an oriented plain arc. As usual, let $ p_1, \ldots, p_d $ denote the intersection points of $ \gamma $ with $T = \{\tau_1, \ldots, \tau_n\}$, and let $\tau_{i_j}$ denote the arc in $T$ containing $p_j$. We associate a tile $G_j$ to each $p_j$ as follows. If $\tau_{i_j}$ is not the folded side of a self-folded triangle in $T$ then $G_j$ is defined as in Definition \ref{relativeorientation}, and is called an \textit{\textbf{ordinary tile}}. Otherwise, the \textit{\textbf{non-ordinary tile}} $G_j$ is defined by glueing two copies of the triangle $(\tau_{i_{j}}, \tau_{i_{j-1}} = \tau_{i_{j+1}}, \tau_{i_j})$ along $\tau_{i_{j}}$, such that the labels on the North and West (equivalently South and East) edges of $G_j$ are equal.  As usual, in both cases, the diagonal $\tau_{i_j}$ is not considered an edge of $G_j$.

\end{defn}

\begin{defn}
\label{snakegraph}

The \textit{\textbf{snake graph}} $G_{\gamma,T}$ associated to a directed plain arc $\gamma$ is defined as follows:

\begin{itemize}

\item If $G_j$ and $G_{j+1}$ are both ordinary tiles then they are glued as in Definition \ref{arcsnakegraph}.

\item If $G_j$ is a non-ordinary tile then it is glued to $G_{j-1}$ and $G_{j+1}$ according to the arrangement illustrated in Figure \ref{nonordinarytile}.

\end{itemize}

For ordinary tiles we have the notion of \textit{relative orientation} given in Definition \ref{relativeorientation}. This notion is extended to each non-ordinary tile $G_j$ by demanding $rel(G_j) := -rel(G_{j-1})$. Note that the equality $rel(G_i) = -rel(G_{i-1})$ will consequently hold for all tiles appearing in $G_{\gamma,T}$.

\end{defn}

\begin{figure}[H]
\begin{center}
\includegraphics[width=12cm]{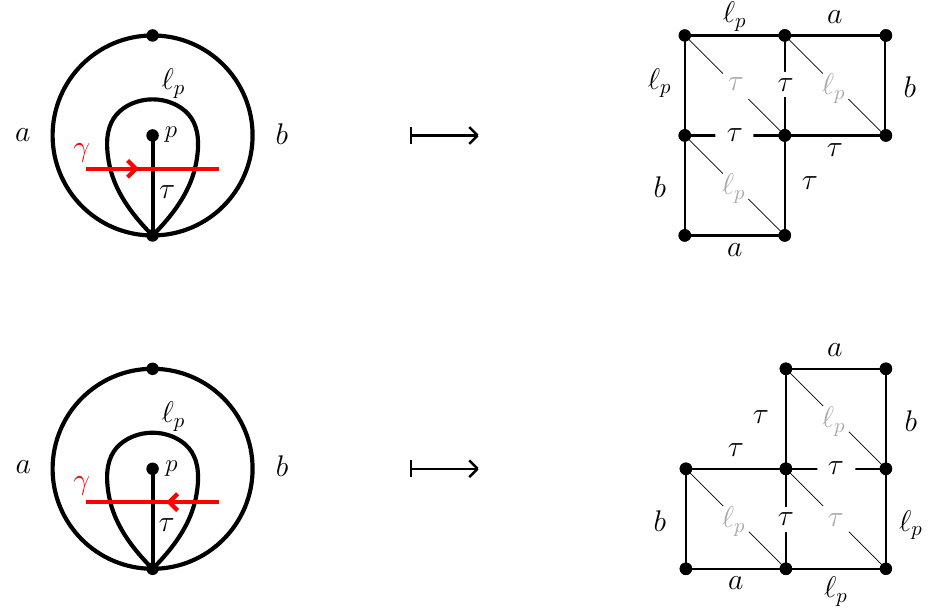}
\caption{The gluing instructions of non-ordinary tiles in a snake graph, see Definition \ref{snakegraph}.}
\label{nonordinarytile}
\end{center}
\end{figure}

\subsection{Loop graphs associated to tagged arcs}

To associate loop graphs to tagged arcs it will first be helpful to introduce the notion of a \textit{hook}. This was also used by Labardini-Fragoso in the context of representations of quivers with potentials arising from surfaces [Section 6.3, \cite{labardini2010quivers}].

\begin{defn}

Let $\gamma$ be a directed tagged arc and let $T$ be an ideal triangulation. If an endpoint of $\gamma$ is incident to a puncture $p$, we define the \textbf{\textit{hook}} at this endpoint to be the curve which:

\begin{itemize}

\item travels around $p$ (clockwise or anticlockwise) intersecting each incident arc at $p$ in $T$ exactly once (up to multiplicity of degree), and then follows $\gamma$, if $\gamma$ starts at the endpoint;

\item or follows $\gamma$ and then travels around $p$ (clockwise or anticlockwise) intersecting each incident arc at $p$ in $T$ exactly once (up to multiplicity of degree), if $\gamma$ ends at the endpoint.

\end{itemize}

\end{defn}

\begin{defn}

Let $\gamma$ be a tagged arc and let $T$ be an ideal triangulation. The associated \textit{\textbf{hooked arc}} $\overset{{}_{\backsim}}{\gamma}$ is obtained by replacing each notched endpoint of $\gamma$ with a hook (see Figure \ref{hook}). Furthermore, for technical reasons, if $\gamma^{\circ} \in T$ then we demand that the tip(s) of the hook(s) intersect $\gamma^{\circ}$ first. Moreover, in this case where $\gamma^{\circ} \in T$, we define the \textit{\textbf{truncated hooked arc}} of $\gamma$ to be resulting curve where the tip(s) have been removed (meaning the truncated version has no intersections with $\gamma$ -- see Figure \ref{truncatedhook}). We use this notion when defining the crossing monomial.

\end{defn}

\begin{figure}[H]
\begin{center}
\includegraphics[width=10cm]{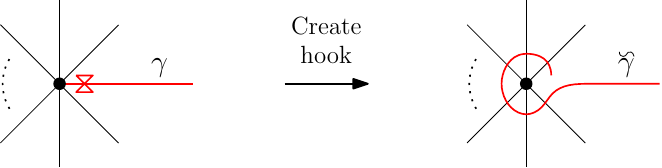}
\caption{A hook associated to a notched endpoint. The arcs in $T$ incident to the puncture determine where the hook ends.}
\label{hook}
\end{center}
\end{figure}

\begin{figure}[H]
\begin{center}
\includegraphics[width=10cm]{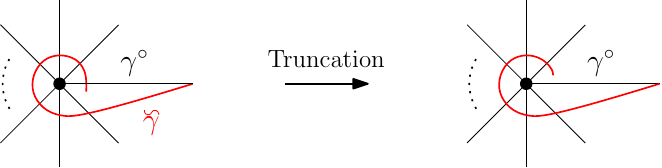}
\caption{We illustrate the truncation process for a tagged arc $\gamma$ whose underlying plain arc $\gamma^{\circ}$ is in $T$. On the left we show the hook associated to a notched endpoint of $\gamma$, and on the right we show the truncated version.}
\label{truncatedhook}
\end{center}
\end{figure}

\newpage

\begin{defn}
\label{surfaceloopgraph}

Let $\gamma$ be a tagged arc, let $T^{\circ}$ be an ideal triangulation. Moreover, as usual, let $\gamma^{\circ}$ denote the underlying plain arc of $\gamma$. In what follows, without loss of generality, we may suppose $\gamma$ is not notched at any puncture enclosed by a self-folded triangle in $T$. \newline 

\noindent We define $G_{\gamma,T}$, the \textit{\textbf{loop graph of $\gamma$ with respect to $T$}}, as follows: \begin{itemize}
\item If $\gamma^{\circ} \notin T$ then consider the snake graph $G_{\overset{{}_{\backsim}}{\gamma},T} = (G_1,\ldots, G_d)$ obtained from the associated hooked arc $\overset{{}_{\backsim}}{\gamma}$. Note that $G_{\gamma^{\circ},T} = (G_{k_1},\ldots, G_{k_2})$ is a subgraph of $G_{\overset{{}_{\backsim}}{\gamma},T}$ for some $k_1,k_2 \in \{1,\ldots, d\}$. We define $G_{\gamma,T}$ to be the loop graph of $G_{\overset{{}_{\backsim}}{\gamma},T}$ where there are loops with respect to $G_1$ and $G_{k_1}$, and with respect to $G_d$ and $G_{k_2}$.
\item If $\gamma^{\circ} \in T$, and $\gamma$ is singly notched at a puncture $p$, then we define $G_{\gamma,T}$ to be the snake graph $G_{\ell_p,T}$, where $\ell_p$ is the unique plain arc enclosing $\gamma$ in a monogon with puncture $p$.
\item If $\gamma^{\circ} \in T$, and $\gamma$ is doubly notched at punctures $p$ and $q$ (not necessarily distinct), then consider the snake graph $G_{\overset{{}_{\backsim}}{\gamma},T} = (G_1,\ldots, G_d)$ obtained from the associated hooked arc $\overset{{}_{\backsim}}{\gamma}$. For notational convenience, let $\tau := \gamma^{\circ}$ (then we also have $\gamma = \tau^{(pq}$). Note that $G_{\tau^{(q)},T} = (G_2,\ldots, G_{k})$ and $G_{\tau^{(p)},T} = (G_{k+1},\ldots, G_d)$ are subgraphs of $G_{\overset{{}_{\backsim}}{\gamma},T}$ for some $k \in \{1,\ldots, d\}$. We define $G_{\gamma,T}$ to be the loop graph of $G_{\overset{{}_{\backsim}}{\gamma},T}$ where there is a loop with respect to $G_1$ and $G_{k+1}$, a loop with respect to $G_d$ and $G_{k}$, and a band with respect to $G_2$ and $G_{d-1}$.
\end{itemize}
Moreover, any loop graph $G^{\bowtie}$ arising in this way is called a \textit{\textbf{surface loop graph}}.

\end{defn}

\begin{rmk}

Note that the hooked arc $\overset{{}_{\backsim}}{\gamma}$ is not unique, since there is a choice of orientation around the puncture at each notched endpoint. However, up to isomorphism, the resulting loop graph $G_{\gamma,T}$ is independent of this choice as the gluing accounts for going ``both ways" around the puncture.

\end{rmk}

\begin{figure}[H]
\begin{center}
\includegraphics[width=14cm]{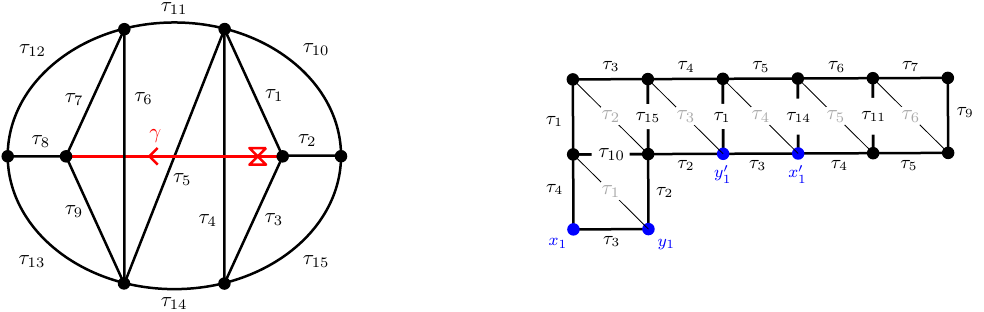}
\caption{An example of the loop graph associated to a singly notched arc. The edge connecting $x_1$ and $y_1$ is glued to the edge connecting $x_1'$ and $y_1'$, where $x_1$ (resp. $y_1$) is glued to $x_1'$ (resp. $y_1'$). Finally, note that the specific hook chosen here is the one travelling clockwise through $\tau_1, \tau_2, \tau_3$.}
\label{single}
\end{center}
\end{figure}

\newpage

\section{Expansion formulae for tagged arcs via loop graphs}
\label{loopgraphsection}

Let $T = \{\tau_1, \ldots, \tau_n\}$ be an ideal triangulation with corresponding tagged triangulation $\iota(T) = \{\tau_1', \ldots, \tau_n'\}$. Let $x_{\tau_1'}, \ldots, x_{\tau_n'}$ be the cluster variables corresponding to $\iota(T)$; when the context is clear we write $x_1, \ldots, x_n$. In a similar fashion we associate variables to $T$: if $\tau_i \in T$ encloses a punctured monogon then it encloses $\tau_i'$ and $\tau_j'$, for some $j$, and we set $x_{\tau_i} := x_ix_j$, otherwise we set $x_{\tau_i} := x_i$. For the (principal) coefficients we set $y_{\tau_i} := y_{\tau_i'} := y_i$ for each $i$.

\begin{defn}
\label{crossingmonomial}

Let $T = \{\tau_1, \ldots, \tau_n\}$ be an ideal triangulation and let $\gamma$ be a directed tagged arc. \newline

If the underlying plain arc $\gamma^{\circ}$ is not in $T$ then we denote by $\tau_{i_1}, \ldots, \tau_{i_d}$ the sequence of arcs in $T$ which the associated hooked arc $\overset{{}_{\backsim}}{\gamma}$ intersects. The \textit{\textbf{crossing monomial}} of $\gamma$ with respect to $T$ is defined as:  \vspace{-1.5mm} $$\text{cross}(\gamma, T) := {\displaystyle \prod_{j=1}^{d} x_{\tau_{i_j}}}.$$

If $\gamma^{\circ} \in T$ and $\gamma$ is a singly or doubly notched arc, then we denote by $\tau_{i_1}, \ldots, \tau_{i_d}$ the sequence of arcs in $T$ which the associated truncated hooked arc intersects. In this setting, the crossing monomial of $\gamma$ with respect to $T$ is defined as:

$$\text{cross}(\gamma, T) := x_{\gamma}\left({\displaystyle \prod_{j=1}^{d} x_{\tau_{i_j}}}\right).$$

\end{defn}

\begin{defn}

Let $T = \{\tau_1, \ldots, \tau_n\}$ be an ideal triangulation and let $ \gamma $ be a directed tagged arc with associated loop graph $ G_{\gamma,T} $. For a good matching $ P$ of $G_{\gamma,T} $ we define the \textit{\textbf{ weight monomial}} $x(P) $ as follows:  \vspace{-0.5mm}

$$ x(P) := {\displaystyle \prod_{\tau_i \in P} x_{\tau_i}}. $$

\end{defn}

\begin{defn}

Let $T$ be an ideal triangulation and let $\gamma$ be a directed tagged arc with associated loop graph $G_{\gamma,T}$. Recall that every good matching $P$ of $G_{\gamma,T}$ induces an orientation on the diagonals of each tile in $G_{\gamma,T}$. We say the diagonal of a tile $G_j$ is \textit{\textbf{positive}} with respect to some $P$ if:

\begin{itemize}

\item the diagonal of $G_j$ is oriented `down' and $rel(G_j)=1$, or

\item the diagonal of $G_j$ is oriented `up' and $rel(G_j)= -1$,

\end{itemize}

\noindent and we say the diagonal is \textit{\textbf{negative}} otherwise.

\end{defn}

\begin{exmp}

In Figure \ref{induceddiagonal}, if we suppose $rel(G_1) = 1$, then the second, third, fourth and fifth diagonals are positive diagonals, and the remaining diagonals are negative with respect to $P$.

\end{exmp}

\newpage

\begin{defn}
\label{goodmatchingcoefficientmonomial}
Let $T = \{\tau_1, \ldots, \tau_n\}$ be an ideal triangulation and let $\gamma$ be a directed tagged arc. For any good matching $P$ of the loop graph $G_{\gamma,T} $ we define the \textit{\textbf{coefficient monomial}}, $y(P) $, as follows:

\begin{equation}
y(P) := \displaystyle \Big(\prod_{\substack{\text{$\tau_{i_j}$ is a } \\ \text{positive}\\ \text{diagonal }}}  y_{\tau_{i_j}} \Big) \Big(\prod_{\substack{\text{$\tau_{i_j}$ is a radius } \\ \text{and a positive}\\ \text{diagonal}}} y_{\tau_{i_j}^{(p)}}^{-1} \Big)
\end{equation}

\noindent where $\tau_{i_j}^{(p)} \in \iota(T)$ is obtained from the plain (radius) arc $\tau_{i_j} \in T$ by adding a notch at the associated puncture $p$.

\end{defn}

\begin{rmk}

Note that, for the convenience of this paper, $ y(P)$ is defined using the language of positive diagonals (first found in \cite{schiffler2010cluster} under the name of $\gamma$-oriented diagonals), rather than the symmetric difference of $P $ with the minimal matching $ P_{-}$. However, as discussed in \cite{musiker2010cluster}, one can easily move between the two. 

\end{rmk}

We now state the main theorem of this paper, which is proven in Section \ref{proof of theorem}.

\begin{thm}
\label{loopexpansion}
Let $(S,M)$ be a bordered surface, and let $T^{\circ}$ be an ideal triangulation with corresponding tagged triangulation $T= \iota(T^\circ)$. We define $\mathcal{A}$ to be the associated cluster algebra with principal coefficients with respect to $\Sigma_T = (\mathbf{x}_T,\mathbf{y}_T,B_T)$. \newline Let $\gamma$ be a tagged arc of $(S,M)$ (for technical reasons, if $\gamma$ is doubly notched then we also assume $(S,M)$ is not twice-punctured and closed). Then the Laurent expansion of $x_{\gamma} \in \mathcal{A}$ with respect to $\Sigma_T$ is given by:

\begin{equation}
\label{loopexpansionformula}
x_{\gamma} = \frac{1}{\text{cross}(\gamma, T^{\circ})}\sum_{P} x(P)y(P),
\end{equation}

\noindent where the sum is taken over all good matchings $P$ of the loop graph $G_{\gamma,T}$. 

\end{thm}

\begin{rmk}
\label{mainremark}
In this paper, the proof of equation (\ref{loopexpansionformula}) relies on the expansion formulae of Musiker, Schiffler and Williams which were not proven in full generality \cite{musiker2011positivity}. Namely, they did not prove the validity of their expansion formulae in the case of a doubly notched arc $\gamma$ on a twice-punctured closed surface $(S,M)$ (with respect to an ideal triangulation $T$ where neither of the underlying singly notched arcs associated to $\gamma$ are in $T$) -- we point the reader to the proof of Theorem 12.9 in \cite{musiker2011positivity} to understand why this setting was untreated. Therefore, as indicated, our Theorem \ref{loopexpansion} does not claim to cover this case either. However, the strategy employed in \cite{geiss2023bangle} avoids the technicalities faced by Musiker, Schiffler and Williams, and can be used to show equation (\ref{loopexpansionformula}) holds in full generality. Namely, the Combinatorial Key Lemma [Theorem 5.3, \cite{geiss2023bangle}] appearing there also holds for tagged arcs (not just closed curves). The decisive difference in approaches comes from flipping perspectives; Musiker, Schiffler and Williams fixed the triangulation, and altered the curve/tagged arc, whereas the application of the Combinatorial Key Lemma works by 
fixing a curve/tagged arc on the surface and altering the underlying triangulation through the process of cluster mutation (showing that expansion formulae continue to hold after mutation). Since the exchange graph is connected (except in the case of a once-punctured closed surface) [Theorem 7.11, \cite{fomin2008cluster}], the later approach naturally proves the validity of expansion formulae for all tagged arcs. \newline
\indent After the release of the first preprint of this paper in 2020, there has been significant interest in the translation of equation (\ref{loopexpansionformula}) into a statement about posets, rather than loop graphs.
Similar in style to this paper, in 2021, Min Huang directly re-packaged the perfect matching framework of Musiker, Schiffler and Williams \cite{musiker2011positivity} into statements about posets -- obtaining a poset version of equation (\ref{loopexpansionformula}) \cite{huang2021new}.
In 2022, Oğuz and Yıldırım found that, by considering this poset perspective, they were able to describe a computationally efficient method of producing the expansion formulae via products of 2 by 2 matrices \cite{ouguz2025cluster} (following a similar style to Musiker and Williams in \cite{musiker2013matrix}).
This poset framework then naturally arose in Weng's work in 2023, which was motivated by finding combinatorial descriptions of cluster DT transformations \cite{weng2023f}.
Finally, later that year, Pilaud, Reading, Schroll presented an alternative proof of (the poset version of) equation (\ref{loopexpansionformula}) through the lens of hyperbolic geometry  [Theorem 1.1, \cite{pilaud2023posets}]. 
We clarify here that, as shown in Corollary \ref{maincor}, the (surface type) poset expansion formulae arising in all four of these papers are directly recovered from equation (\ref{loopexpansionformula}) and Section \ref{latticesection}. Namely, with respect to a fixed initial tagged triangulation $T$ of a bordered surface $(S,M)$, then for each tagged arc $\gamma$ of $(S,M)$, the poset $P_{\gamma}$ they consider is precisely the poset $Q_{G_{\gamma,T}}$ defined in Definition \ref{WrittenHasseQuiverDefinition}.
\end{rmk}

\begin{exmp}
Consider the singly notched arc $\gamma$ and (ideal) triangulation $T$ found in Figure \ref{single}. To obtain the cluster variable $x_{\gamma}$ with respect to principal coefficients at $T$, Theorem \ref{loopexpansion} tells us we must compute the crossing monomial $\text{cross}(\gamma, T)$ and enumerate all good matchings of the associated loop graph (also found in Figure \ref{single}). We see $$\text{cross}(\gamma, T) = x_1x_2x_3x_4x_5x_6$$ and the complete collection of good matchings is provided in Figure \ref{singlelattice}. We thus obtain: \vspace{-5mm} 

\begin{gather*}
\begin{aligned}
x_{\gamma}
&=
\frac{1}{x_1 x_2 x_3 x_4 x_5 x_6}
\Big(
  x_1 x_2 x_4^2 x_6 x_9
+ x_1^2 x_4 x_6 x_9 x_{15}\textcolor{red}{y_3}
+ x_1 x_2 x_4 x_9 x_{11} x_{14}\textcolor{red}{y_5}  \\
&\quad
+ x_1 x_3 x_4 x_6 x_9 x_{10}\textcolor{red}{y_2 y_3}
+ x_1^2 x_9 x_{11} x_{14} x_{15}\textcolor{red}{y_3 y_5}
+ x_1 x_2 x_4 x_5 x_7 x_{14}\textcolor{red}{y_5 y_6} \\
&\quad
+ x_1 x_3 x_5 x_9 x_{11} x_{15}\textcolor{red}{y_3 y_4 y_5}
+ x_1 x_3 x_9 x_{10} x_{11} x_{14}\textcolor{red}{y_2 y_3 y_5}
+ x_1^2 x_5 x_7 x_{14} x_{15}\textcolor{red}{y_3 y_5 y_6} \\
&\quad
+ x_3^2 x_5 x_9 x_{10} x_{11}\textcolor{red}{y_2 y_3 y_4 y_5}
+ x_1 x_3 x_5^2 x_7 x_{15}\textcolor{red}{y_3 y_4 y_5 y_6}
+ x_1 x_3 x_5 x_7 x_{10} x_{14}\textcolor{red}{y_2 y_3 y_5 y_6} \\
&\quad
+ x_2 x_3 x_4 x_5 x_9 x_{11}\textcolor{red}{y_1 y_2 y_3 y_4 y_5}
+ x_3^2 x_5^2 x_7 x_{10}\textcolor{red}{y_2 y_3 y_4 y_5 y_6}
+ x_2 x_3 x_4 x_5^2 x_7\textcolor{red}{y_1 y_2 y_3 y_4 y_5 y_6}
\Big).
\end{aligned}
\end{gather*}

\end{exmp}

\begin{exmp}
The reader may find it helpful to compare the loop graph expansion formulae with Musiker, Schiffler and Williams' method. To this end, we now follow the doubly-notched example found in [Section 5.3, \cite{musiker2011positivity}], which we have also illustrated in Figure \ref{double}. For this choice of $\gamma^{(pq)}$ and $T$, the collection of good matchings of $G_{\gamma^{(pq)},T}$ is shown in Figure \ref{doublelattice}. Moreover, since $$\text{cross}(\gamma^{(pq)}, T) = x_3x_4x_5x_6x_7x_8$$ Theorem \ref{loopexpansion} gives us the following expansion of $x_{\gamma^{(pq)}}$ with respect to principal coefficients at $T$: \vspace{-5mm} 

\begin{gather*}
\begin{aligned}
x_{\gamma^{(pq)}}
&=
\frac{1}{x_3 x_4 x_5 x_6 x_7 x_8}
\Big(
  x_3 x_4 x_6^2 x_8
+ x_4^2 x_6 x_8 x_{10}\textcolor{red}{y_5}
 + x_3 x_4 x_6 x_8 x_9\textcolor{red}{y_7} \\
&\quad
+ x_2 x_4 x_5 x_6 x_8\textcolor{red}{y_3 y_5}
+ x_4^2 x_8 x_9 x_{10}\textcolor{red}{y_5 y_7}
+ x_2 x_4 x_5 x_8 x_9\textcolor{red}{y_3 y_5 y_7} \\
&\quad
+ x_4 x_5 x_7 x_9 x_{10}\textcolor{red}{y_5 y_6 y_7}
+ x_2 x_5^2 x_7 x_9\textcolor{red}{y_3 y_5 y_6 y_7}
+ x_4 x_5 x_6 x_7 x_{10}\textcolor{red}{y_5 y_6 y_7 y_8} \\
&\quad
+ x_3 x_5 x_6 x_7 x_9\textcolor{red}{y_3 y_4 y_5 y_6 y_7}
+ x_2 x_5^2 x_6 x_7\textcolor{red}{y_3 y_5 y_6 y_7 y_8}
+ x_3 x_5 x_6^2 x_7\textcolor{red}{y_3 y_4 y_5 y_6 y_7 y_8}
\Big).
\end{aligned}
\end{gather*}

\end{exmp}

\begin{figure}[H]
\begin{center}
\includegraphics[width=14cm]{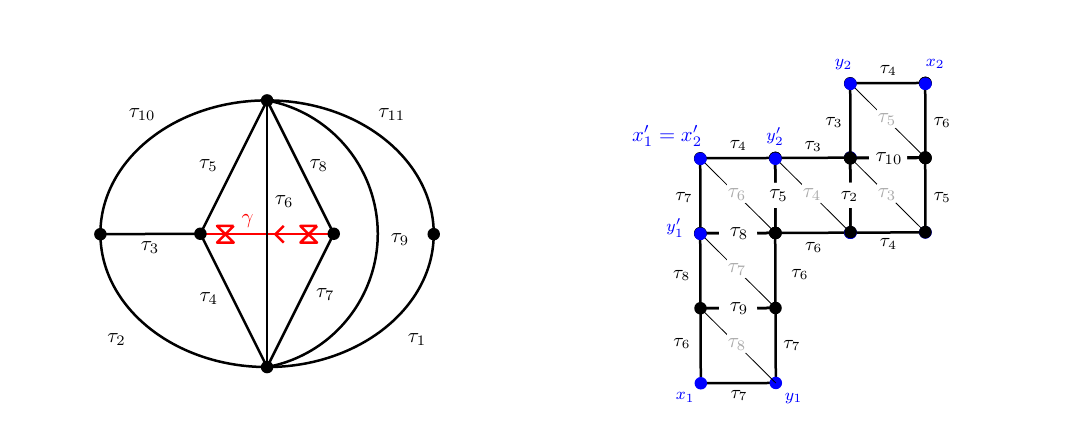}\vspace{-5mm} 
\caption{An example of the loop graph associated to a doubly notched arc $\gamma^{(pq)}$ whose underlying arc $\gamma$ is not in the triangulation. The edge connecting $x_i$ and $y_i$ is glued to the edge connecting $x_i'$ and $y_i'$, for $i \in \{1,2\}$. In this gluing $x_i$ (resp. $y_i$) is glued to $x_i'$ (resp. $y_i'$).}
\label{double}
\end{center}
\end{figure}

\begin{exmp}\vspace{-5mm} 
Consider the doubly notched arc $\gamma^{(pq)}$ and (ideal) triangulation $T$ found in Figure \ref{DoubleUnderlying}. To obtain the cluster variable $x_{\gamma^{(pq)}}$ with respect to principal coefficients at $T$, Theorem \ref{loopexpansion} tells us we must compute the crossing monomial $\text{cross}(\gamma^{(pq)}, T)$ and enumerate all good matchings of the associated loop graph (also found in Figure \ref{DoubleUnderlying}). We see $$\text{cross}(\gamma^{(pq)}, T) = x_3x_4x_5x_7x_8x_{\gamma}$$ and the complete collection of good matchings is provided in Figure \ref{DoubleLatticeUnderlying}. We thus obtain: \vspace{-5mm}

\begin{gather*}
\begin{aligned}
x_{\gamma^{(pq)}}
&=
\frac{1}{x_3 x_4 x_5 x_7 x_8 x_{\gamma}}
\Big(
  x_3 x_4 x_5 x_7 x_8
+ x_3 x_4^2 x_8^2\textcolor{red}{y_{\gamma}}
+ x_3x_4x_8 x_9 x_{\gamma}\textcolor{red}{y_7 y_{\gamma}} \\
&\quad
+ x_4^2 x_8 x_{10} x_{\gamma} \textcolor{red}{y_5 y_{\gamma}}
+  x_4 x_9 x_{10} x_{\gamma}^2\textcolor{red}{y_5 y_7 y_{\gamma}}
+  x_2 x_4 x_5 x_8 x_{\gamma} \textcolor{red}{y_3 y_5 y_{\gamma}} \\
&\quad
+ x_4 x_5 x_7 x_{10} x_{\gamma}\textcolor{red}{y_5 y_7 y_8 y_{\gamma}}
+ x_2 x_5 x_9 x_{\gamma}^2 \textcolor{red}{y_3 y_5 y_7 y_{\gamma}}
+ x_2 x_5^2 x_7 x_{\gamma}\textcolor{red}{y_3 y_5 y_7 y_8 y_{\gamma}} \\
&\quad
+ x_3 x_5 x_7 x_9 x_{\gamma}\textcolor{red}{y_3 y_4 y_5 y_7 y_{\gamma}} 
+ x_3 x_5^2 x_7^2 \textcolor{red}{y_3 y_4 y_5 y_7 y_8 y_{\gamma}}
+ x_3 x_4 x_5 x_7 x_8 \textcolor{red}{y_3 y_4 y_5 y_7 y_8 y_{\gamma}^2}
\Big).
\end{aligned}
\end{gather*}

\end{exmp}

\begin{figure}[H]
\begin{center}
\includegraphics[width=14cm]{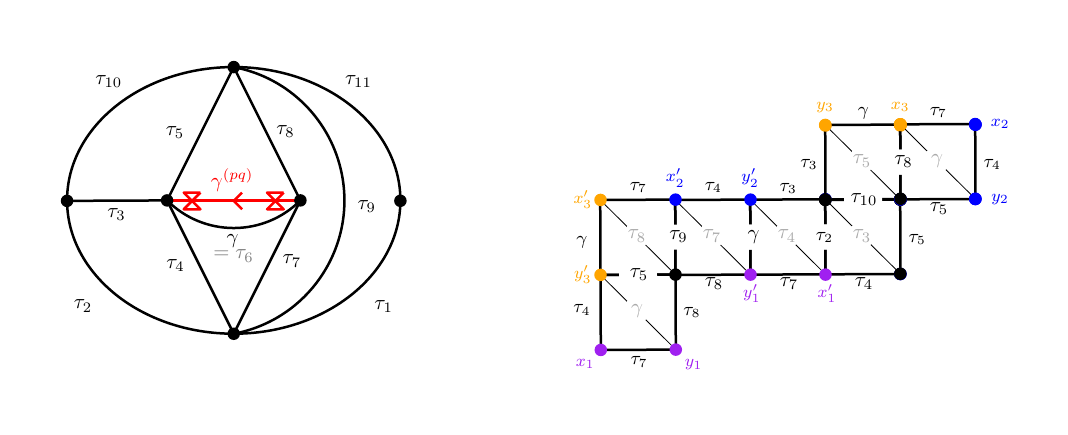}\vspace{-6mm} 
\caption{An example of the loop graph associated to a doubly notched arc $\gamma^{(pq)}$ whose underlying arc $\gamma$ is in the triangulation. The edge connecting $x_i$ and $y_i$ is glued to the edge connecting $x_i'$ and $y_i'$, for $i \in \{1,2,3\}$. In this gluing $x_i$ (resp. $y_i$) is glued to $x_i'$ (resp. $y_i'$).}
\label{DoubleUnderlying}
\end{center}
\end{figure}

\begin{figure}[H]
\begin{center}
\includegraphics[width=14cm]{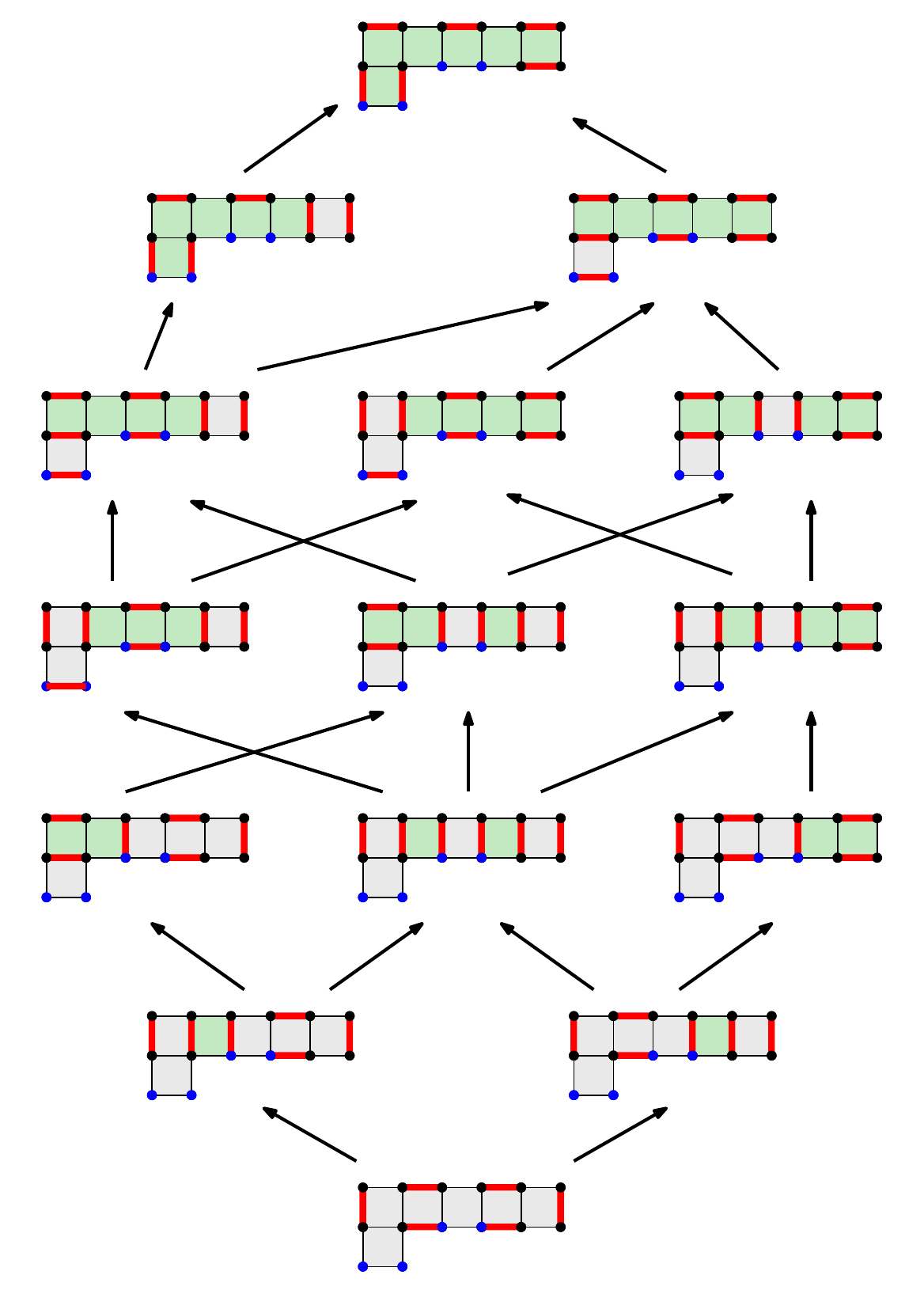}
\caption{The lattice of good matchings of the loop graph found in Figure \ref{single}. Tiles with positive diagonals are shaded in green.}
\label{singlelattice}
\end{center}
\end{figure}

\begin{figure}[H]
\begin{center}
\includegraphics[width=10cm]{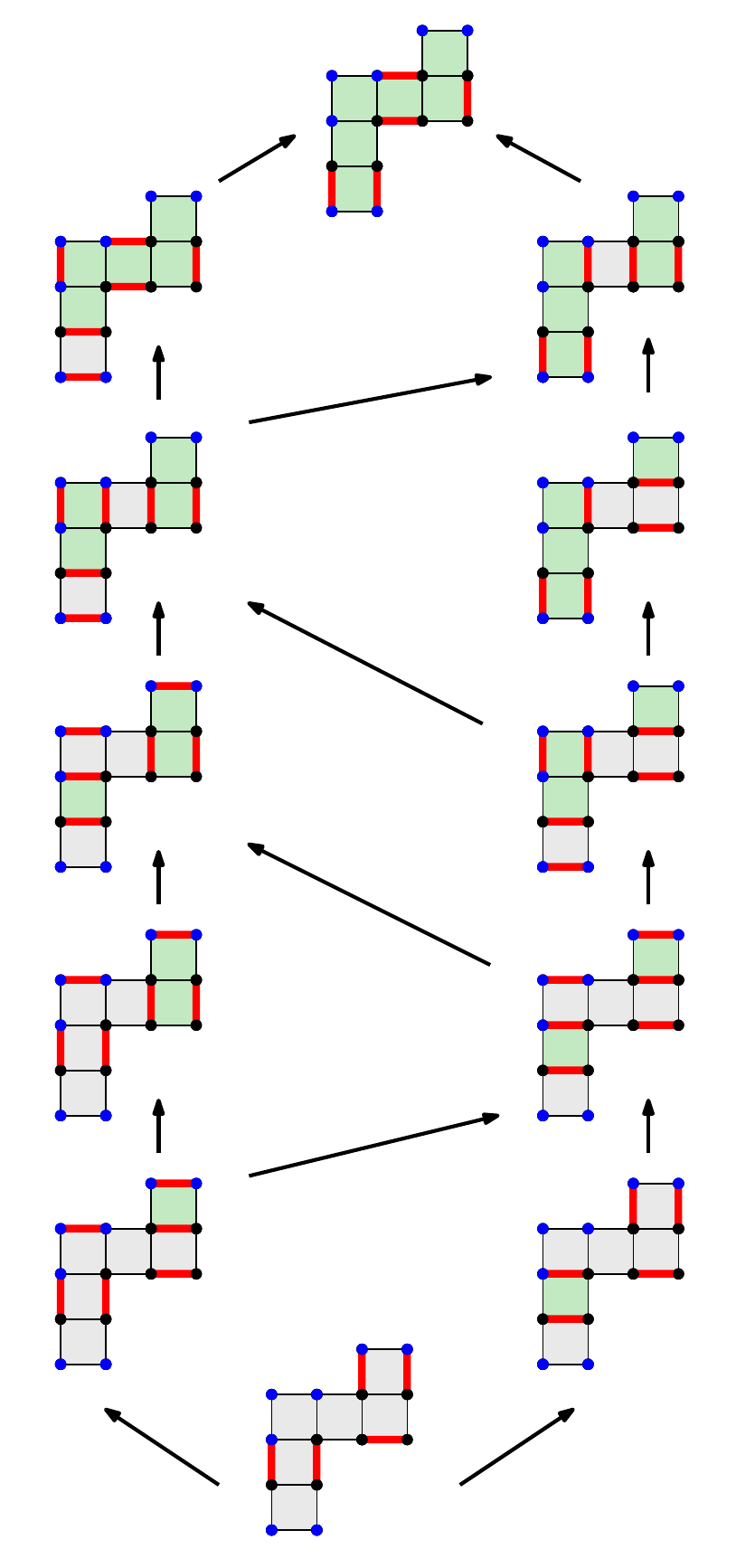}
\caption{The lattice of good matchings of the loop graph found in Figure \ref{double}. Tiles with positive diagonals are shaded in green.}
\label{doublelattice}
\end{center}
\end{figure}

\begin{figure}[H]
\begin{center}
\includegraphics[width=14cm]{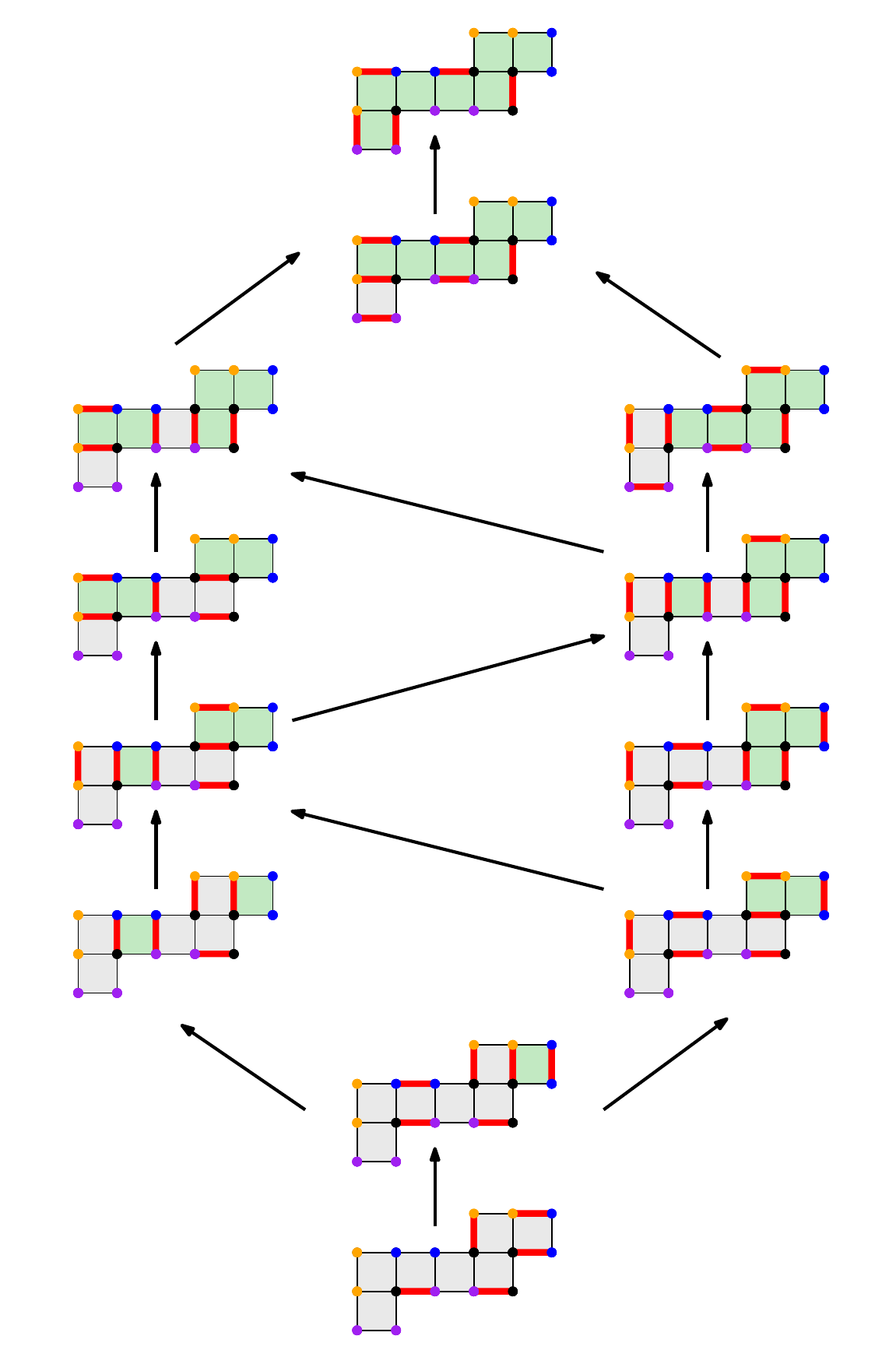}
\caption{The lattice of good matchings of the loop graph found in Figure \ref{DoubleUnderlying}. Tiles with positive diagonals are shaded in green.}
\label{DoubleLatticeUnderlying}
\end{center}
\end{figure}

\section{Proof of Theorem \ref{loopexpansion}}
\label{proof of theorem}

We prove Theorem \ref{loopexpansion} by reinterpreting Musiker, Schiffler and Williams' language of `$\gamma$-symmetric perfect matchings' and `$\gamma$-compatible pairs of $\gamma$-symmetric perfect matchings' as statements about good matchings of loop graphs.

\subsection{$\gamma$-symmetric perfect matchings}

\begin{defn}
\label{singlynotcheddef}
Let $ \gamma^{(p)} $ be a tagged arc which has one end notched at a puncture $p$, and its other end tagged plain at another marked point. We call $\gamma^{(p)} $ a \textbf{\textit{singly notched arc}} at $ p$ and denote its underlying plain arc by $\gamma$. Furthermore, we denote by $\ell_p$ the unique plain arc enclosing $\gamma$ in a monogon with puncture $p$.
\end{defn}

Note that if $ \gamma$ has $k $ intersection points with $T ^{\circ}$ and $\gamma^{\circ} \notin T$, then $\ell_p $ has $ 2k + l$ intersection points for $l \geq 2 $. Specifically, $ l$ is the degree of the puncture $p$ in $T ^{\circ}$ (the number of ends of arcs in $T^{\circ}$ incident to $p$). See Figure \ref{SinglePunctureIntersections}.

\begin{figure}[H]
\begin{center}
\includegraphics[width=9cm]{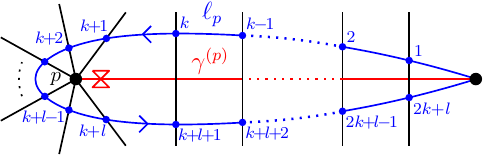}
\caption{An illustration of a singly notched arc $\gamma^{(p)}$ and the associated plain arc $\ell_p$ enclosing $\gamma^{(p)}$ in a monogon with puncture $p$. An underlying ideal triangulation $T^{\circ}$ is denoted in black, and with respect to the chosen orientation on $\ell_p$, we label the intersection points of $\ell_p$ with $T^{\circ}$ by $1 \ldots, 2k+ l$ (where $k$ is the number of intersections of $\gamma$ with $T^{\circ}$, and $l$ is the degree of $p$ in $T^{\circ}$.}
\label{SinglePunctureIntersections}
\end{center}
\end{figure}

\begin{defn}

Consider the snake graph $ G_{\ell_p,T^{\circ}} = (G_1, \ldots, G_{2k +l}) $ and let $$ \mathcal{H}_{\gamma, 1} := (G_1,\ldots,G_d)\setminus \{NE(G_k)\} $$ be the subgraph of $ G_{\gamma, 1} := (G_1, \ldots, G_k) $ where the North-East vertex of $ G_k $ and its incident edges $N(G_k), E(G_k)$ have been removed.

Similarly, let $$ \mathcal{H}_{\gamma, 2} := (G_{k+l+1}, \ldots ,G_{2k+l}) \setminus \{SW(G_{k+l+1})\} $$ be the subgraph of $ G_{\gamma, 2} := (G_{k+l+1},\ldots ,G_{2k+l}) $ where the South-West vertex of $ G_{k+l+1} $ and its incident edges $ S(G_{k+l+1}), W(G_{k+l+1})$ have been removed. \newline \indent

Note that $$ \mathcal{H}_{\gamma, 1} \cong \mathcal{H}_{\gamma, 2} \hspace{11mm} \text{and}  \hspace{11mm}  G_{\gamma, 1}  \cong G_{\gamma, 2} \cong  G_{\gamma, T^{\circ}}.$$

\end{defn}

\begin{defn}

A perfect matching $ P $ of $ G_{\ell_p,T^{\circ}} $ is said to be \textbf{\textit{$ \gamma $-symmetric}} if: $$ P_{\vert_{\mathcal{H}_{\gamma, 1}}}  \cong P_{\vert_{\mathcal{H}_{\gamma, 2}}} $$ with respect to the isomorphism $ \mathcal{H}_{\gamma, 1}  \cong  \mathcal{H}_{\gamma, 2} $.

\end{defn}

\begin{defn}

Let $ P $ be a $ \gamma $-symmetric perfect matching of $ G_{\ell_p,T^{\circ}} $. The associated \textit{\textbf{weight monomial}} and \textit{\textbf{coefficient monomial}} are defined respectively, as follows: $$ \overline{x}(P) := \frac{x(P)}{x(P_{\vert_{G_{\gamma, i}}})}  \hspace{10mm} \text{and}  \hspace{10mm} \overline{y}(P) := \frac{y(P)}{y(P_{\vert_{G_{\gamma, i}}})}. $$

The index $ i \in \{1,2\} $ above is chosen such that $ P_{\vert_{G_{\gamma, i}}} $ is a perfect matching of $ G_{\gamma, i} $ -- this is well defined by [Lemma 12.4, \cite{musiker2011positivity}].

\end{defn}

\begin{thm}[Theorem 4.17. \cite{musiker2011positivity}]
\label{singlynotched}
Let $(S,M)$ be a bordered surface with puncture $ p $. Let $ T^{\circ} $ be an ideal triangulation with corresponding tagged triangulation $ T = \iota(T^{\circ}) $. Let $\mathcal{A} $ be the associated cluster algebra with principal coefficients with respect to $\Sigma_T = (\mathbf{x}_T,\mathbf{y}_T,B_T)$. Suppose $\gamma^{(p)}$ is a singly notched arc at $p$ whose underlying plain arc $\gamma$ is not in $T$, and that $p$ is not the puncture of a self-folded triangle in $T$. Then the Laurent expansion of $x_{\gamma^{(p)}} \in \mathcal{A}$ with respect to $\Sigma_T$ is given by:

\begin{equation}
x_{\gamma^{(p)}} = \frac{\text{cross}(\gamma,T^{\circ})}{\text{cross}(\ell_p,T^{\circ})} {\displaystyle \sum_{P} \overline{x}(P)\overline{y}(P)}
\end{equation}

\noindent where the sum is over all $ \gamma $-symmetric perfect matchings $ P $ of the snake graph $ G_{\ell_p, T^{\circ}} $.

\end{thm}

\begin{prop}
\label{singlebijection}
Let $T$ be an ideal triangulation and let $ \gamma^{(p)} $ be a singly notched arc at $p$ whose underlying plain arc $ \gamma $ is not in $ T $, and suppose that $ p $ is not the puncture of a self-folded triangle in $ T $. Then there exists a bijection 

\begin{align*}
 \Phi \hspace{1mm}: \hspace{2mm} &\Big\{ \stackanchor{\text{\hspace{0.5mm} $ \gamma $-symmetric perfect matchings \hspace{0.5mm}}}{\text{$P$ of $G_{\ell_p, T}$}}\Big\} \hspace{5mm} \longrightarrow \hspace{5mm} \Big\{ \text{Good matchings of $G_{\gamma^{(p)}, T}$}\Big\}
\end{align*}

such that $$\overline{x}(P) = x(\Phi(P)) \hspace{5mm} and \hspace{5mm} \overline{y}(P) = y(\Phi(P))$$ for all $\gamma$-symmetric perfect matchings $P$ of $G_{\ell_p, T}$.

\end{prop}

\begin{proof}
Following the conventions of this section, let us write $G_{\ell_p,T} = (G_1, \ldots, G_{2k +l})$. Note that $G_{\gamma^{(p)}, T}$ can be obtained from $(G_1, \ldots, G_{k + l})$ by creating a loop with respect to $G_{k+l}$ and $G_k$, and can also be obtained from $(G_{k+1}, \ldots, G_{2k + l})$ by creating a loop with respect to $G_{k+1}$ and $G_{k+l+1}$. Furthermore, note that precisely one of the subgraphs $(G_{k-1},G_k,G_{k+1})$ and $(G_{k+l},G_{k+l+1},G_{k+l+2})$ is a zig-zag. Without loss of generality, we will suppose $(G_{k-1},G_k,G_{k+1})$ is zig-zag and that $E(G_k)$ and $S(G_k)$ are boundary edges. Finally, note that it is possible that $\ell_p$ has only $l+2$ intersections with arcs in $T$, in which case $G_{k-1} = \emptyset = G_{k+l+2}$. The proof described below works in the same way for this case too.  \newline \indent

Let $P$ be a $\gamma$-symmetric perfect matching of $G_{\ell_p, T}$. From our assumptions we see that $P$ will contain one of the (boundary) edges $E(G_k)$ or $S(G_k)$. \newline

\noindent \textbf{Case 1}: $P$ involves $E(G_k)$. See Figure \ref{Prop6.6East}. \newline \indent  Note that $P$ involves the edge $E(G_k)$ \textit{if and only if} $$P = E(G_k) \cup P_1 \cup P_1' \cup P_2$$ where \begin{itemize}

\item $P_1$ and $P_1'$ are perfect matchings of $(G_1,\ldots, G_{k-1})$ and $(G_{k+l+2},\ldots, G_{2k+l})$, respectively, such that $P_1 \cong P_1'$,

\item $P_2$ is a perfect matching of $(G_{k+2},\ldots, G_{k+l})$.

\end{itemize}

Similarly, a good matching $P$ of $G_{\gamma^{(p)}, T} = (G_1, \ldots, G_{k + l})^{\bowtie}$ is a right or centre cut at $E(G_k)$ \textit{if and only if} $$P = P_1 \cup P_2$$ where $P_1$ is a perfect matching of $(G_1,\ldots, G_{k-1})$ and $P_2$ is a perfect matching of $(G_{k+2},\ldots, G_{k+l})$. \newline

For the second part of the proposition, recall that $\overline{x}(P) := \frac{x(P)}{x(P_{\vert_{G_{\gamma,1}}})}$. Since $P_{\vert_{G_{\gamma,1}}} = E(G_k) \cup P_1$ and $P_1 \cong P_1'$ then $$\overline{x}(P) = \frac{x(E(G_k))x(P_1)x(P_1')x(P_2)}{x(E(G_k))x(P_1)} = x(P_1')x(P_2) = x(P_1 \cup P_2).$$

Moreover, with respect to the $\gamma$-symmetric perfect matching $P$, a diagonal of a tile in $G_{\gamma,1} := (G_1,\ldots, G_k)$ is positive \textit{if and only if} the diagonal of the corresponding tile in $G_{\gamma,2} := (G_{k+l+1},\ldots, G_{2k+l})$ is positive. Hence $$\overline{y}(P) := \frac{y(P)}{y(P_{\vert_{G_{\gamma,1}}})} = y(P_1\cup P_2).$$

\noindent \textbf{Case 2}: $P$ involves $S(G_k)$. See Figure \ref{Prop6.6South}. \newline \indent  Analogous to above, $P$ involves the edge $S(G_k)$ \textit{if and only if} $$P = S(G_k) \cup P_1 \cup P_1' \cup P_2 \cup e_1 \cup e_2$$ where \begin{itemize}

\item $P_1$ and $P_1'$ are perfect matchings of $(G_1,\ldots, G_{k-1})\setminus E(G_{k-1})$ and $(G_{k+l+2},\ldots, G_{2k+l})\setminus b$, respectively, such that $P_1 \cong P_1'$ (here $b= W(G_{k+l+2})$ if $l$ is odd, and $b= S(G_{k+l+2})$ if $l$ is even),

\item $P_2$ is a perfect matching of $(G_{k+1},\ldots, G_{k+l-1})$.

\item $e_1 = N(G_{k+l+1})$ and $e_2 = S(G_{k+l+1})$ if $l$ is odd, and $e_1 = E(G_{k+l+1})$ and $e_2 = W(G_{k+l+1})$ if $l$ is even.

\end{itemize}

Similarly, a good matching $P$ of $G_{\gamma^{(p)}, T} = (G_1, \ldots, G_{k + l})^{\bowtie}$ is a left cut at $E(G_k)$ \textit{if and only if} $$P = S(G_k) \cup P_1 \cup P_2$$ where $P_1$ is a perfect matching of $(G_1,\ldots, G_{k-1})\setminus E(G_{k-1})$ and $P_2$ is a perfect matching of $(G_{k+1},\ldots, G_{k+l-1})$. \newline

For the second part of the proposition, recall that $\overline{x}(P) := \frac{x(P)}{x(P_{\vert_{G_{\gamma,2}}})}$. Since $P_{\vert_{G_{\gamma,2}}} = P_1' \cup e_1 \cup e_2$ then $$\overline{x}(P) = \frac{x(S(G_k))x(P_1)x(P_1')x(P_2)x(e_1)x(e_2)}{x(P_1')x(e_1)x(e_2)} = x(S(G_k) \cup P_1 \cup P_2).$$

Moreover, as in Case 1, a diagonal of a tile in $G_{\gamma,1} := (G_1,\ldots, G_k)$ is positive \textit{if and only if} the diagonal of the corresponding tile in $G_{\gamma,2} := (G_{k+l+1},\ldots,G_{2k+1})$ is positive. Hence $$\overline{y}(P) := \frac{y(P)}{y(P_{\vert_{G_{\gamma,2}}})} = y(S(G_k) \cup P_1\cup P_2).$$

This completes the proof as any good matching of $G_{\gamma^{(p)}, T} = (G_1, \ldots, G_{k + l})^{\bowtie}$ has a right, left or centre cut at $E(G_k)$.

\end{proof}

\begin{figure}[H]
\begin{center}
\includegraphics[width=13cm]{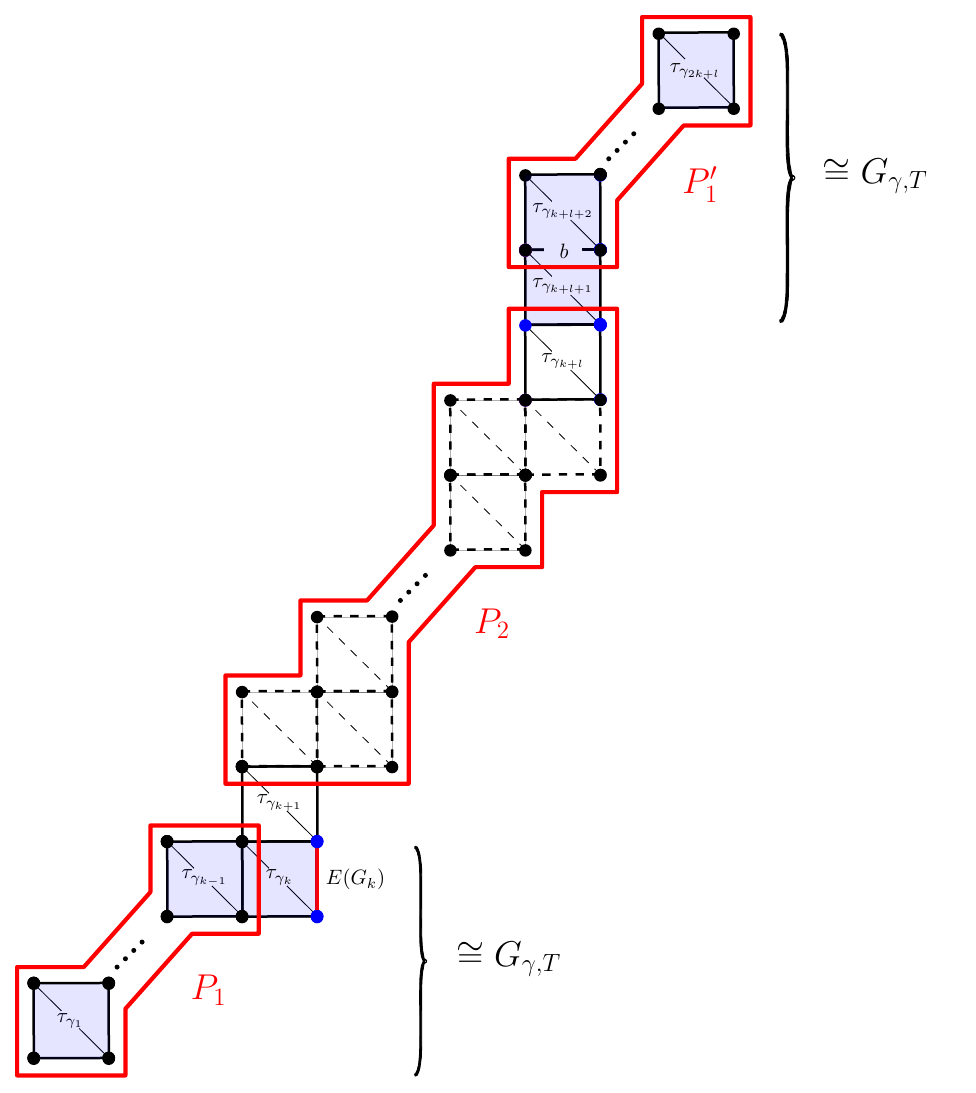}
\caption{A schematic proof of Proposition \ref{singlebijection} -- case 1. Namely, there is a bijection between $\gamma$-symmetric matchings $P$ of the entire snake graph $G_{\ell_p,T} = (G_1, \ldots, G_{2k +l})$ such that $E(G_k) \in P$ and good matchings $P^{\bowtie}$ of the loop graph $G_{\gamma^{(p)},T} = (G_1, \ldots, G_{k +l})^{\bowtie}$ such that $P^{\bowtie}$ has a right or centre cut at $E(G_k)$. In particular, this bijection is given by: $P = P_1\cup \{E(G_k)\} \cup P_2 \cup P_1' \leftrightarrow P^{\bowtie} = P_1\cup P_2$. }
\label{Prop6.6East}
\end{center}
\end{figure}

\begin{figure}[H]
\begin{center}
\includegraphics[width=13cm]{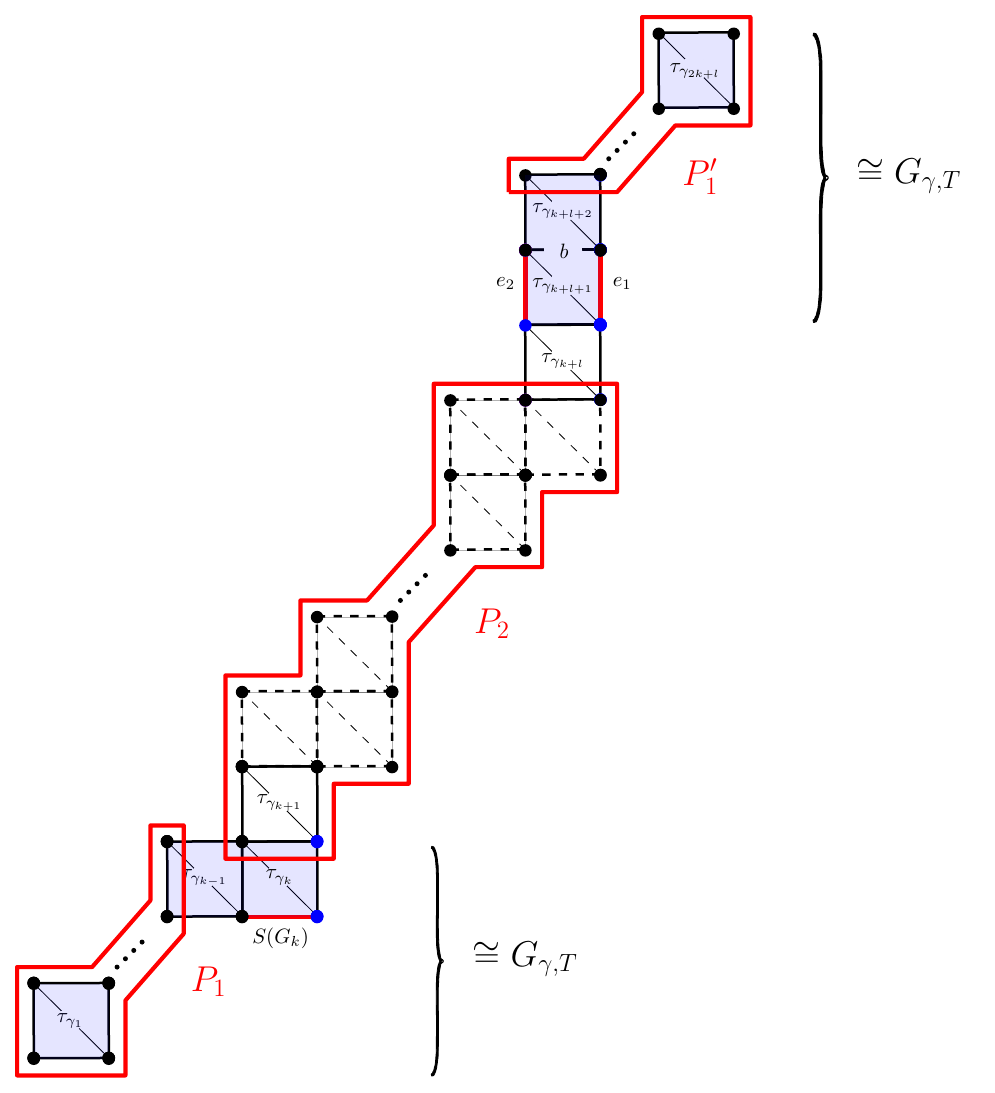}
\caption{A schematic proof of Proposition \ref{singlebijection} -- case 2. Namely, there is a bijection between $\gamma$-symmetric matchings $P$ of the entire snake graph $G_{\ell_p,T} = (G_1, \ldots, G_{2k +l})$ such that $S(G_k) \in P$ and good matchings $P^{\bowtie}$ of the loop graph $G_{\gamma^{(p)},T} = (G_1, \ldots, G_{k +l})^{\bowtie}$ such that $P^{\bowtie}$ has a left cut at $E(G_k)$. In particular, this bijection is given by: $P = P_1\cup \{S(G_k)\} \cup P_2 \cup \{e_1\} \cup \{e_2\} \cup P_1' \leftrightarrow P^{\bowtie} = P_1 \cup \{S(G_k)\} \cup  P_2.$}
\label{Prop6.6South}
\end{center}
\end{figure}

\newpage

The following observation will be helpful in the next section, where we compare good matchings and compatible pairs of $\gamma$-symmetric perfect matchings. It follows directly from the above proof.

\begin{cor}
\label{proof cor} Let $P$ be a $\gamma$-symmetric perfect matching of $G_{\ell_p, T}$. With respect to the set-up described in the opening paragraph of the proof of Proposition \ref{singlebijection}, the following holds. \newline \indent

If $E(G_k) \in P$ then $P$ restricts to a perfect matching on $(G_1,\ldots, G_{k+l})$. Moreover, we get a bijection between $\gamma$-symmetric perfect matchings of $G_{\ell_p,T}$ involving $E(G_k)$ and good matchings of $$G_{\gamma^{(p)},T} = (G_1,\ldots, G_{k+l})^{\bowtie}$$ with right or centre cut at $E(G_k)$ via $$P \mapsto P_{\vert_{(G_1,\ldots,G_{k+l})}}\setminus E(G_k).$$

Similarly, if $S(G_k) \in P$ then $P$ restricts to a perfect matching on $(G_{k+1},\ldots, G_{2k+l})$. Moreover, we get a bijection between $\gamma$-symmetric perfect matchings of $G_{\ell_p,T}$ involving $S(G_k)$ and good matchings of $$G_{\gamma^{(p)},T} = (G_1,\ldots, G_{k+l})^{\bowtie} \cong (G_{k+1},\ldots, G_{2k+l})^{\bowtie}$$ with left cut at $E(G_k)$ via $$P \mapsto P_{\vert_{(G_{k+1},\ldots,G_{2k+l})}}\setminus b$$

\noindent where $b = S(G_{k+l+1})$ if $l$ is odd and $b = W(G_{k+l+1})$ if $l$ is even.

\end{cor}

\subsection{Compatible pairs of $\gamma$-symmetric perfect matchings}

\begin{defn}

Let $\gamma^{(pq)}$ be a tagged arc which is notched at puncture $p$ and $q$ (we allow $p = q$). We say $\gamma^{(pq)} $ is a \textbf{\textit{doubly notched arc}} at $p $ and $q $. We define $\ell_p$ and $\ell_q$ as in Theorem \ref{singlynotched}, however, note that when $ p= q $ then $ \ell_p$ and $ \ell_q$ are not strictly arcs as they are self-intersecting. \end{defn}

\begin{defn}
\label{compatiblepair}
Let $\gamma^{(pq)} $ be a doubly notched arc and let $ P_p$ and $ P_q$ be $ \gamma $-symmetric perfect matchings of $ G_{\ell_p, T^{\circ}} $ and $ G_{\ell_q,  T^{\circ}} $, respectively. \newline \indent The pair $ (P_p, P_q) $ is called \textit{\textbf{$\gamma$-compatible}} if the following holds for some $ i,j \in \{1,2\}$:

\begin{itemize}

\item ${P_p}_{\vert_{G_{\gamma, i}}}$ and ${P_q}_{\vert_{G_{\gamma, j}}}$ are perfect matchings of $G_{\gamma, i} \subset G_{\ell_p, T^{\circ}}$ and $G_{\gamma, j} \subset G_{\ell_q, T^{\circ}}$, respectively,

\item and ${P_p}_{\vert_{G_{\gamma, i}}} \cong {P_q}_{\vert_{G_{\gamma, j}}}$ with respect to the canonical isomorphism $G_{\gamma, i} \cong G_{\gamma, j}$.
 
\end{itemize} 

\end{defn}

\begin{defn}

Let $(P_p,P_q)$ be a $\gamma$-compatible pair of $\gamma$-symmetric perfect matchings of $(G_{\ell_p,T^{\circ}},G_{\ell_q,T^{\circ}})$. The associated \textit{\textbf{cluster monomial}} and \textit{\textbf{coefficient monomial}}, are defined, respectively, as follows: $$\overline{\overline{x}}(P_p ,P_q ) := \frac{x(P_p )x(P_q )}{x({P_p }_{\vert_{G_{\gamma, i }}})^3} \hspace{10mm} \text{and} \hspace{10mm} \overline{\overline{y}}(P_p ,P_q ) := \frac{y(P_p )y(P_q )}{y({P_p }_{\vert_{G_{\gamma, i }}})^3}.$$

Where the index $ i \in \{1,2\} $ is chosen such that $ {P_p}_{\vert_{G_{\gamma, i}}} \cong { P_q}_{\vert_{G_{\gamma, j }}} $ for some $ j \in \{1,2\} $, as in Definition \ref{compatiblepair}. This is well defined by [Lemma 12.4, \cite{musiker2011positivity}].

\end{defn}

\begin{thm}[Theorem 4.20, \cite{musiker2011positivity}]
\label{doublynotched}
Let $(S,M)$ be a bordered surface with punctures $p$ and $q$, that is not twice punctured and closed. Let $T^{\circ}$ be an ideal triangulation with corresponding tagged triangulation $T$. Let $\mathcal{A}$ be the associated cluster algebra with principal coefficients with respect to $\Sigma_T = (\mathbf{x}_T,\mathbf{y}_T,B_T)$. Suppose $\gamma^{(pq)}$ is a doubly notched arc at $p$ and $q$ whose underlying plain arc $\gamma$ is not in $T^{\circ}$, and that neither $p$ nor $q$ are the puncture of a self-folded triangle in $T$. Then the Laurent expansion of $x_{\gamma^{(pq)}} \in \mathcal{A}$ with respect to $\Sigma_T$ is given by:

\begin{equation}
 x_{(\gamma^{(pq)})} = \frac{\text{cross}(\gamma,T^{\circ})^3}{\text{cross}(\ell_p,T^{\circ})\text{cross}(\ell_q,T^{\circ})} {\displaystyle \sum_{(P_p,P_q)} \overline{\overline{x}}(P_p,P_q)\overline{\overline{y}}(P_p,P_q)}
\end{equation}

Where the sum is over all $\gamma$-compatible pairs of $\gamma$-symmetric perfect matchings $(P_p,P_q)$ of the pair of snake graphs $(G_{\ell_p, T^{\circ}},G_{\ell_q, T^{\circ}})$.

\end{thm}

\begin{prop}
\label{doublebijection}
Let $T$ be an ideal triangulation and let $\gamma^{(pq)}$ be a doubly notched arc at $p$ and $q$ whose underlying plain arc $\gamma$ is not in $T$. Furthermore, suppose that neither $p$ nor $q$ are the puncture of a self-folded triangle in $T$. Then there exists a bijection 

\begin{align*}
 \Phi \hspace{1mm}: \hspace{2mm} &\Bigg\{ \substack{\text{ \normalsize $\gamma$-compatible pairs of $\gamma$-symmetric} \\ \text{ \normalsize perfect matchings $(P_p,P_q)$}\\ \text{ \normalsize of $(G_{\ell_p, T},G_{\ell_p, T})$}}\Bigg\} \hspace{5mm} \longrightarrow \hspace{5mm} \Big\{ \text{Good matchings of $G_{\gamma^{(pq)}, T}$}\Big\}
\end{align*}

such that $$\overline{\overline{x}}(P_p,P_q) = x(\Phi(P_p,P_q)) \hspace{5mm} and \hspace{5mm} \overline{\overline{y}}(P_p,P_q) = y(\Phi(P_p,P_q))$$ for all $\gamma$-compatible pairs of $\gamma$-symmetric perfect matchings $(P_p,P_q)$ of $(G_{\ell_p, T},G_{\ell_p, T})$.

\end{prop}

\begin{proof}

Let us write $G_{\ell_p,T} = (G_1, \ldots, G_{2k +l_p})$ and $G_{\ell_q,T} = (H_1, \ldots, H_{2k + l_q})$. Furthermore, following the same reasoning used in the proof of Proposition \ref{singlebijection}, without loss of generality we will suppose both $(G_{k-1},G_k,G_{k+1})$ and $(H_{k-1},H_k,H_{k+1})$ are zig-zag and that $E(G_k)$,$S(G_k)$,$E(H_k)$,$S(H_k)$ are boundary edges.

The collection of all $\gamma$-compatible pairs of $\gamma$-symmetric perfect matchings $(P_p,P_q)$ can be decomposed into the disjoint union of the following four classes: \begin{enumerate}

\item $E(G_k) \in P_p$ and $E(H_k) \in P_q$.

\item $S(G_k) \in P_p$ and $E(H_k) \in P_q$.

\item $E(G_k) \in P_p$ and $S(H_k) \in P_q$.

\item $S(G_k) \in P_p$ and $S(H_k) \in P_q$.

\end{enumerate}

For each class there exist $i,j \in \{1,2\}$ such that $P_p$ and $P_q$ restrict to perfect matchings on $G_{\gamma,i} \subset G_{\ell_p,T}$ and $G_{\gamma,j} \subset G_{\ell_q,T}$, respectively, for all $(P_p,P_q)$ in the chosen class. Moreover, by definition of $\gamma$-compatibility, these restrictions are isomorphic with respect to the isomorphism $G_{\gamma,i} \cong G_{\gamma,j}$.

This means that $P_p$ and $P_q$ can be glued along $G_{\gamma,i}$ and $G_{\gamma,j}$. Applying Corollary \ref{proof cor}, for each of the four classes above, we get a bijection between the following classes of good matchings $P$ of $G_{\gamma,T}$:

\begin{enumerate}

\item $P$ has a right or centre cut with respect to $E(G_k)$ and a right or centre cut with respect to $E(H_k)$.

\item $P$ has a left cut with respect to $E(G_k)$ and a right or centre cut with respect to $E(H_k)$.

\item $P$ has a right or centre cut with respect to $E(G_k)$ and a left cut with respect to $E(H_k)$.

\item $P$ has a left cut with respect to $E(G_k)$ and a left cut with respect to $E(H_k)$.

Finally, coupling this with Proposition \ref{singlebijection}, the bijection $(P_p,P_q) \mapsto P$ described above satisfies $$\overline{\overline{x}}(P_p,P_q) = x(P) \hspace{10mm} \text{and} \hspace{10mm} \overline{\overline{y}}(P_p,P_q) = y(P).$$

\end{enumerate}

\end{proof}

\subsection{Proof of Theorem \ref{loopexpansion} when $\gamma^{\circ} \notin T$}

We now prove the Main Theorem \ref{loopexpansion} in the case when the underlying plain arc $\gamma^{\circ}$ is not in the initial triangulation $T$.

\begin{proof}[Proof of Theorem \ref{loopexpansion} when $\gamma^{\circ} \notin T$.]\leavevmode\par \vspace{3mm}

When $\gamma$ is a plain arc then $G_{\gamma,T}$ is just a snake graph and equation (\ref{loopexpansionformula}) reduces to [Theorem 4.10. \cite{musiker2011positivity}]. 

When $\gamma = \gamma^{(p)}$ is a singly notched arc then equation (\ref{loopexpansionformula}) follows from Theorem \ref{singlynotched}, Proposition \ref{singlebijection} and the equality $$\text{cross}(\gamma^{(p)},T^{\circ}) = \frac{\text{cross}(\ell_p,T^{\circ})}{\text{cross}(\gamma^{\circ},T^{\circ})}.$$

Finally, when $\gamma = \gamma^{(pq)}$ is a doubly notched arc then equation (\ref{loopexpansionformula}) follows from Theorem \ref{doublynotched}, Proposition \ref{doublebijection} and the equality $$\text{cross}(\gamma^{(pq)},T) = \frac{\text{cross}(\ell_p,T^{\circ})\text{cross}(\ell_q,T^{\circ})}{\text{cross}(\gamma^{\circ},T^{\circ})^3}.$$

\end{proof}

\subsection{Proof of Theorem \ref{loopexpansion} when $\gamma^{\circ} \in T$}

In this section we prove the Main Theorem \ref{loopexpansion} in the case when the underlying plain arc $\gamma^{\circ}$ is contained in the initial triangulation $T$.
 
Recall that if $\gamma$ is singly notched at a puncture $p$ then $x_{\gamma}x_{\gamma^{\circ}} = \ell_p$ where $\ell_p$ is the plain arc enclosing $\gamma$ in a monogon with puncture $p$.
Moreover, since Musiker, Schiffler and Williams proved that Theorem 5.7 holds for $\ell_p$, then, since $x_{\gamma^{\circ}}$ is an initial cluster variable when $\gamma^{\circ} \in T$, one immediately gets Theorem 5.7 also holds for $x_{\gamma} = \frac{\ell_p}{x_{\gamma^{\circ}}}$ [Remark 4.17, \cite{musiker2011positivity}]. Therefore, the only remaining case to consider is when $\gamma$ is doubly notched.

\newpage

\begin{prop}[Proposition 4.21, \cite{musiker2011positivity}]\label{DoubleUnderlyingExpansion}

Let $(S,M)$ be a bordered surface with punctures $p$ and $q$, that is not twice punctured and closed. Let $T^{\circ}$ be an ideal triangulation with corresponding tagged triangulation $T$. Let $\mathcal{A}$ be the associated cluster algebra with principal coefficients with respect to $\Sigma_T = (\mathbf{x}_T,\mathbf{y}_T,B_T)$. Suppose $\gamma^{(pq)}$ is a doubly notched arc at $p$ and $q$ whose underlying plain arc $\gamma$ is an arc in $T^{\circ}$. Then the Laurent expansion of $x_{\gamma^{(pq)}} \in \mathcal{A}$ with respect to $\Sigma_T$ is given by:

$$x_{\gamma^{(pq)}} = \frac{x_{\gamma^{(p)}}x_{\gamma^{(q)}}y_{\gamma} + \Bigg(1- \displaystyle \prod_{\tau \in T} y_{\tau}^{e_p(\tau)}\Bigg)\Bigg(1- \displaystyle \prod_{\tau \in T} y_{\tau}^{e_q(\tau)}\Bigg)}{x_{\gamma}}$$

where $\gamma^{(p)}$ and $\gamma^{(q)}$ are the singly notched arcs\footnotemark{} at $p$ and $q$ respectively whose underlying plain arc is    $e_p(\tau)$ and $e_q(\tau)$ denote the number of endpoints of $\tau \in T$ incident to $p$ and $q$, respectively.

\end{prop}

\footnotetext{Note that if $\gamma$ is a loop then $\gamma^{(p)}$ and $\gamma^{(q)}$
will not actually be arcs. Nevertheless, using equation~\ref{loopexpansionformula}
to assign a Laurent polynomial to these ``arcs'', Musiker, Schiffler and Williams
show Proposition~\ref{DoubleUnderlyingExpansion} still holds in this setting.}

\begin{figure}[H]
\begin{center}
\includegraphics[width=14cm]{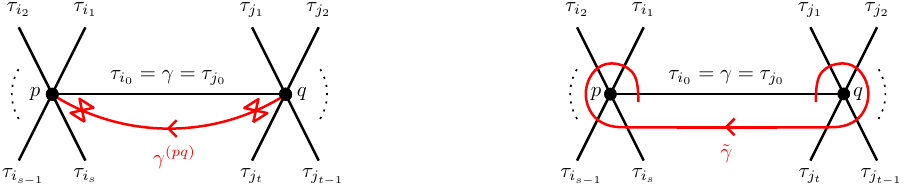}
\caption{On the left we illustrate the generic configuration of an ideal triangulation $T$ and a doubly notched arc $\gamma^{(pq)}$ such that the underlying plain arc $\gamma$ is in $T$. On the right we illustrate an associated hooked arc $\tilde{\gamma}$ of $\gamma^{(pq)}$.}
\label{GenericDoubleUnderlyingSurface}
\end{center} \vspace{-5mm}
\end{figure}

\begin{proof}[Proof of Theorem \ref{loopexpansion} when $\gamma^{\circ} \in T$.]
\leavevmode\par

With respect to the generic configuration labelled in Figure \ref{GenericDoubleUnderlyingSurface}, Proposition \ref{DoubleUnderlyingExpansion} may be equivalently stated as follows, where the summations below are over all perfect matchings of the snake graphs $G_{\ell_p, T^{\circ}}$ and $G_{\ell_q, T^{\circ}}$, respectively: \vspace{-5mm}

\begin{align*}
x_{\gamma^{(pq)}} &=
\frac{
x_{\gamma^{(p)}}x_{\gamma^{(q)}}y_{\gamma}
+ \displaystyle \Bigg(1- \prod_{k=0}^s y_{\tau_{i_k}}\Bigg)
  \Bigg(1- \prod_{k=0}^t y_{\tau_{j_k}}\Bigg)
}{x_{\gamma}} \\
&=
\frac{
\Bigg(\displaystyle \sum_{\substack{P \text{ of }\\ G_{\ell_p, T^{\circ}}}} \hspace{-2mm} x(P)y(P)\hspace{-0.5mm}\Bigg)\hspace{-1mm}
\Bigg(\displaystyle \sum_{\substack{P \text{ of }\\ G_{\ell_q, T^{\circ}}}} \hspace{-2mm} x(P)y(P)\hspace{-0.5mm}\Bigg)
y_{\gamma}
+ \Bigg(1- \prod_{k=0}^s y_{\tau_{i_k}}\hspace{-0.5mm}\Bigg)\hspace{-1mm}
  \Bigg(1- \prod_{k=0}^t y_{\tau_{j_k}}\hspace{-0.5mm}\Bigg) \prod_{k=0}^s x_{\tau_{i_k}}
\prod_{k=0}^t x_{\tau_{j_k}}}{x_{\gamma}\cdot \displaystyle \left( \prod_{k=0}^s x_{\tau_{i_k}}\right)
\left( \prod_{k=0}^t x_{\tau_{j_k}}\right)}
\end{align*}

where we used $$x_{\gamma^{(p)}} =  \displaystyle \frac{1}{\text{cross}(\gamma^{(p)},T^{\circ})} \left( \sum_{P \text{ of } G_{\ell_p, T^{\circ}}} \hspace{-2mm} x(P) y(P) \right)
= \frac{\displaystyle \sum_{P \text{ of } G_{\ell_p, T^{\circ}}} \hspace{-2mm} x(P) y(P)}{\displaystyle \prod_{k=0}^s x_{\tau_{i_k}}}.$$ and the analogous expression for $x_{\gamma^{(q)}}$.

Therefore, to prove Theorem \ref{loopexpansion} in this setting, since $\text{cross}(\gamma^{(pq)},T^{\circ}) := x_{\gamma}\cdot \displaystyle \left( \prod_{k=1}^s x_{\tau_{i_k}}\right)
\left( \prod_{k=1}^t x_{\tau_{j_k}}\right)$, it suffices to show that 

\begin{align}
\displaystyle & x_{\gamma}^2\left(\sum_{\substack{P \text{ of }\\ G_{\gamma^{(pq)}, T}}} x(P)y(P) \right) - \left(\displaystyle \sum_{\substack{P \text{ of }\\ G_{\ell_p, T^{\circ}}}} x(P)y(P)\right) 
\left(\displaystyle \sum_{\substack{P \text{ of }\\ G_{\ell_q, T^{\circ}}}} x(P)y(P) \right)
y_{\gamma} \nonumber \\
 &= \left(\left(1  \ \ + \ \  \prod_{k=0}^s y_{\tau_{i_k}} \prod_{k=0}^t y_{\tau_{j_k}}\right) \ \ - \ \ \left(\prod_{k=0}^s y_{\tau_{i_k}} \ \ + \ \ \prod_{k=0}^t y_{\tau_{j_k}}\right)\right)\left(\displaystyle \prod_{k=0}^s x_{\tau_{i_k}}\right)
\left( \prod_{k=0}^t x_{\tau_{j_k}}\right) \label{rephrasedprop4.21}
\end{align}

where $G_{\gamma^{(pq)}, T}$ is the loop graph of $\gamma^{(pq)}$ with respect to $T$, and the associated sum is over all good matchings $P$ of $G_{\gamma^{(pq)}, T}$. \newline

In view of this, let $P_1$ and $P_2$ be any two perfect matchings of $ G_{\ell_p, T^{\circ}}$ and $ G_{\ell_q, T^{\circ}}$, respectively. Furthermore, as in Figure \ref{GenericDoubleUnderlyingGraph}, let $e_1$ be the unique shared edge between tiles $T_{\tau_{j_t}}$ and $T_{\tau_{i_s}}$ in $G_{\gamma^{(pq)}, T}$, and let $e_2$ and $e_3$ be the glued edges of tiles $T_{\tau_{j_1}}$ and $T_{\tau_{i_1}}$ in $G_{\gamma^{(pq)}, T}$, respectively.

Then $(P_1,P_2)$ naturally descends to a good matching $\overline{(P_1,P_2)}$ of $G_{\gamma^{(pq)}, T}$ so long as the following two conditions are satisfied: $e_1 \in P_1$ or $e_1 \in P_2$; and $e_2 \in P_1$ or $e_3 \in P_2$. Indeed, note that:

\begin{itemize}

\item if $e_1 \in P_1$ and $e_1 \notin P_2$ then $e_3 \in P_2$, hence $\overline{(P_1,P_2)} := P_1\setminus\{e_1\} \cup P_2\setminus \{e_3\}$ is a good matching of $G_{\gamma^{(pq)}, T}$. On the other hand, if $e_1 \in P_2$ and $e_1 \notin P_1$ then $e_2 \in P_1$, hence $\overline{(P_1,P_2)} := P_1\setminus\{e_2\} \cup P_2\setminus \{e_1\}$ is a good matching of $G_{\gamma^{(pq)}, T}$.

\item In a similar fashion, if $e_2 \in P_1$ and $e_3 \notin P_2$ then $e_1 \in P_2$, hence $\overline{(P_1,P_2)} := P_1\setminus\{e_2\} \cup P_2\setminus \{e_1\}$ is a good matching of $G_{\gamma^{(pq)}, T}$. Analogously, if $e_3 \in P_2$ and $e_2 \notin P_1$ then $e_1 \in P_1$, hence $\overline{(P_1,P_2)} := P_1\setminus\{e_1\} \cup P_2\setminus \{e_3\}$ is a good matching of $G_{\gamma^{(pq)}, T}$.

\end{itemize}

Moreover, with respect to the orientation on diagonals in $G_{\gamma^{(pq)}, T}$ induced by $(\overline{P_1,P_2})$, in all four cases above, the diagonal of the first tile is negatively oriented, and the diagonal of the last tile is positively oriented. Therefore, since $e_1$, $e_2$, and $e_3$ are all labelled by $\gamma$, then we have 

\begin{align}
\label{productofpmsums}
x_{\gamma}^2 \cdot x(\overline{P_1,P_2})y(\overline{P_1,P_2}) =  x(P_1)x(P_2)y(P_1)y(P_2) \cdot y_{\gamma}.
\end{align}

It remains to consider the two cases when: $e_1 \notin P_1$ and $e_1 \notin P_2$; or $e_2 \notin P_1$ and $e_3 \notin P_2$. Note that, as utilised above, there is only one perfect matching $P_1'$ of $G_{\ell_p, T^{\circ}}$ such that $e_1 \notin P_1'$, and only one perfect matching $P_2'$ of $G_{\ell_p, T^{\circ}}$ such that $e_1 \notin P_2'$. Likewise, there is only one perfect matching $P_1''$ of $G_{\ell_p, T^{\circ}}$ such that $e_2 \notin P_1''$, and only one perfect matching $P_2''$ of $G_{\ell_p, T^{\circ}}$ such that $e_3 \notin P_2''$.

However, as illustrated in Figure \ref{doubleloopgraphminus}, neither of these pairs descend to good matchings of $G_{\gamma^{(pq)}, T}$. Moreover, we have 
\begin{align}
x(P_1')y(P_1')x(P_2')y(P_2')y_{\gamma} &= \displaystyle \left(\prod_{k=0}^t y_{\tau_{j_k}}\right)\left(\displaystyle \prod_{k=0}^s x_{\tau_{i_k}}\right)
\left( \prod_{k=0}^t x_{\tau_{j_k}}\right) \label{overcounte1} \\
x(P_1'')y(P_1'')x(P_2'')y(P_2'')y_{\gamma} &= \displaystyle \left(\prod_{k=0}^s y_{\tau_{i_k}}\right)\left(\displaystyle \prod_{k=0}^s x_{\tau_{i_k}}\right)
\left( \prod_{k=0}^t x_{\tau_{j_k}}\right) \label{overcounte2e3}
\end{align}

Finally, note that the collection of all good matchings of the form $\overline{(P_1,P_2)}$ does not consist of all good matchings of $G_{\gamma^{(pq)}, T}$. Indeed, the collection fails to include good matchings inducing a positive orientation on the diagonal of the first tile, or a negative orientation on the diagonal of the final tile. As illustrated in Figure \ref{doubleloopgraphplus}, if a good matching $P'$ induces a positive orientation on the diagonal of the first tile, then the zig-zag nature of the graph uniquely defines the matched edges up to tile $T_{\tau_{j_{t-1}}}$ (and the induced orientation of all diagonals up this tile are positive), moreover, due to the loop between the first tile and $T_{\tau_{i_s}}$ then $\tau_{j_t} \in T_{\tau_{i_s}}$ must also be in $P'$. Consequently, the zig-zag nature of the graph (along with the loop at tiles $T_{\tau_{i_1}}$ and $T_{\tau_{j_1}}$, and the loop at $T_{\tau_{j_t}}$ and the final tile) furthermore uniquely determines all remaining matched edges in $P_1$ (and, in turn, the induced orientation of all diagonals of tiles in $G_{\gamma^{(pq)}, T}$ are positive). In particular, we have

\begin{align}
x(P')y(P') = \displaystyle \left(\prod_{k=0}^s y_{\tau_{i_k}}\right)\left(\prod_{k=0}^t y_{\tau_{j_k}}\right)\left(\displaystyle \prod_{k=0}^s x_{\tau_{i_k}}\right)
\left( \prod_{k=0}^t x_{\tau_{j_k}}\right). \label{undercount1}
\end{align}

A similar analysis shows there is a unique good matching $P''$ of $G_{\gamma^{(pq)}, T}$ inducing a negative orientation on the diagonal of the final tile of $G_{\gamma^{(pq)}, T}$ (see Figure \ref{doubleloopgraphplus}), and this yields the equality
\begin{align}
x(P'')y(P'') = \displaystyle \left(\displaystyle \prod_{k=0}^s x_{\tau_{i_k}}\right)
\left( \prod_{k=0}^t x_{\tau_{j_k}}\right).  \label{undercount2}
\end{align}

In summary, equations (\ref{productofpmsums}), (\ref{overcounte1}), (\ref{overcounte2e3}), (\ref{undercount1}), (\ref{undercount2}) deduce equation (\ref{rephrasedprop4.21}). This completes the proof of Theorem \ref{loopexpansion} when $\gamma^{\circ} \in T$.

\end{proof}

\begin{figure}[H]
\begin{center}
\includegraphics[width=17cm]{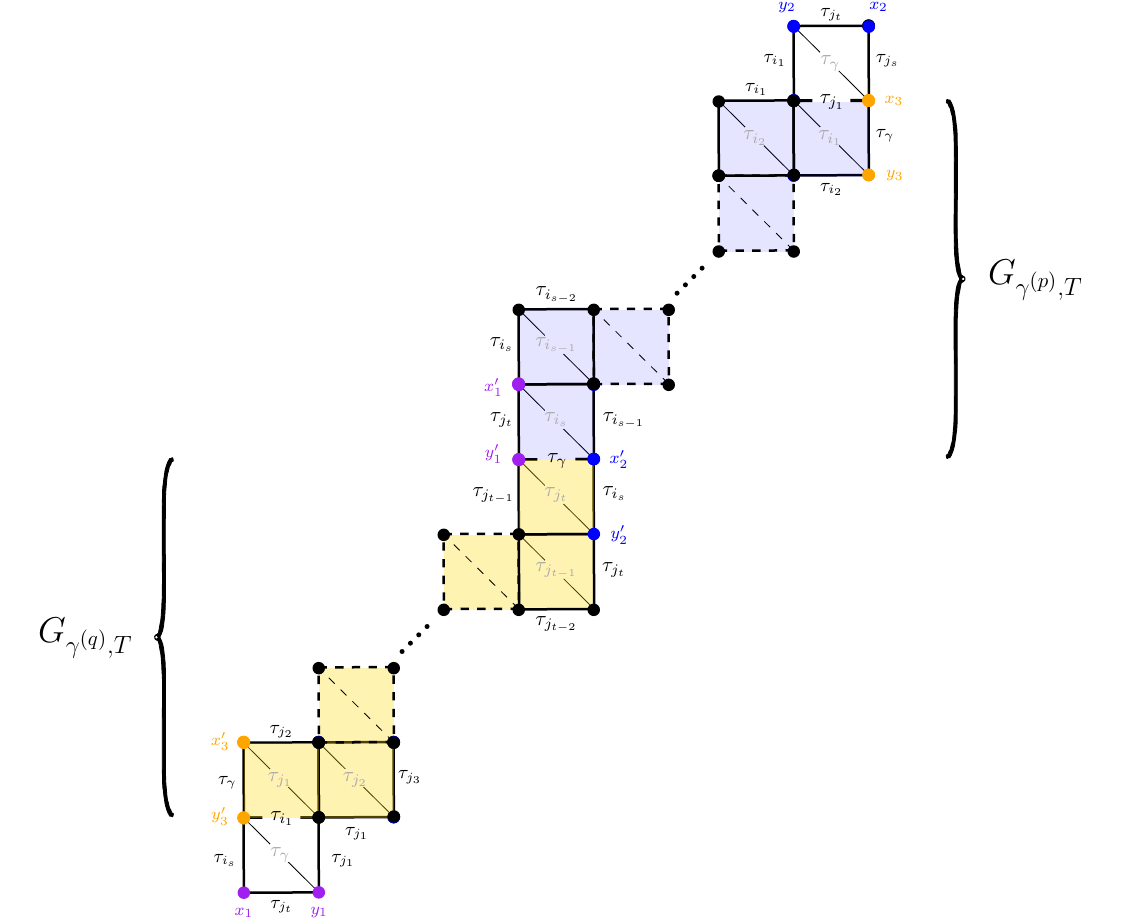}
\caption{We illustrate the associated loop graph $G_{\gamma^{(pq)}, T}$ of the ideal triangulation $T$ and doubly notched arc $\gamma^{(pq)}$ presented in Figure \ref{GenericDoubleUnderlyingSurface}. The shaded yellow tiles represent the sub snake graph of $G_{\gamma^{(pq)}, T}$ isomorphic to $G_{\gamma^{(q)}, T}$. Likewise, the shaded blue tiles highlight the sub snake graph of $G_{\gamma^{(pq)}, T}$ isomorphic to $G_{\gamma^{(p)}, T}$.}
\label{GenericDoubleUnderlyingGraph}
\end{center}
\end{figure}

\begin{figure}[H]
\begin{center}
\includegraphics[width=17cm]{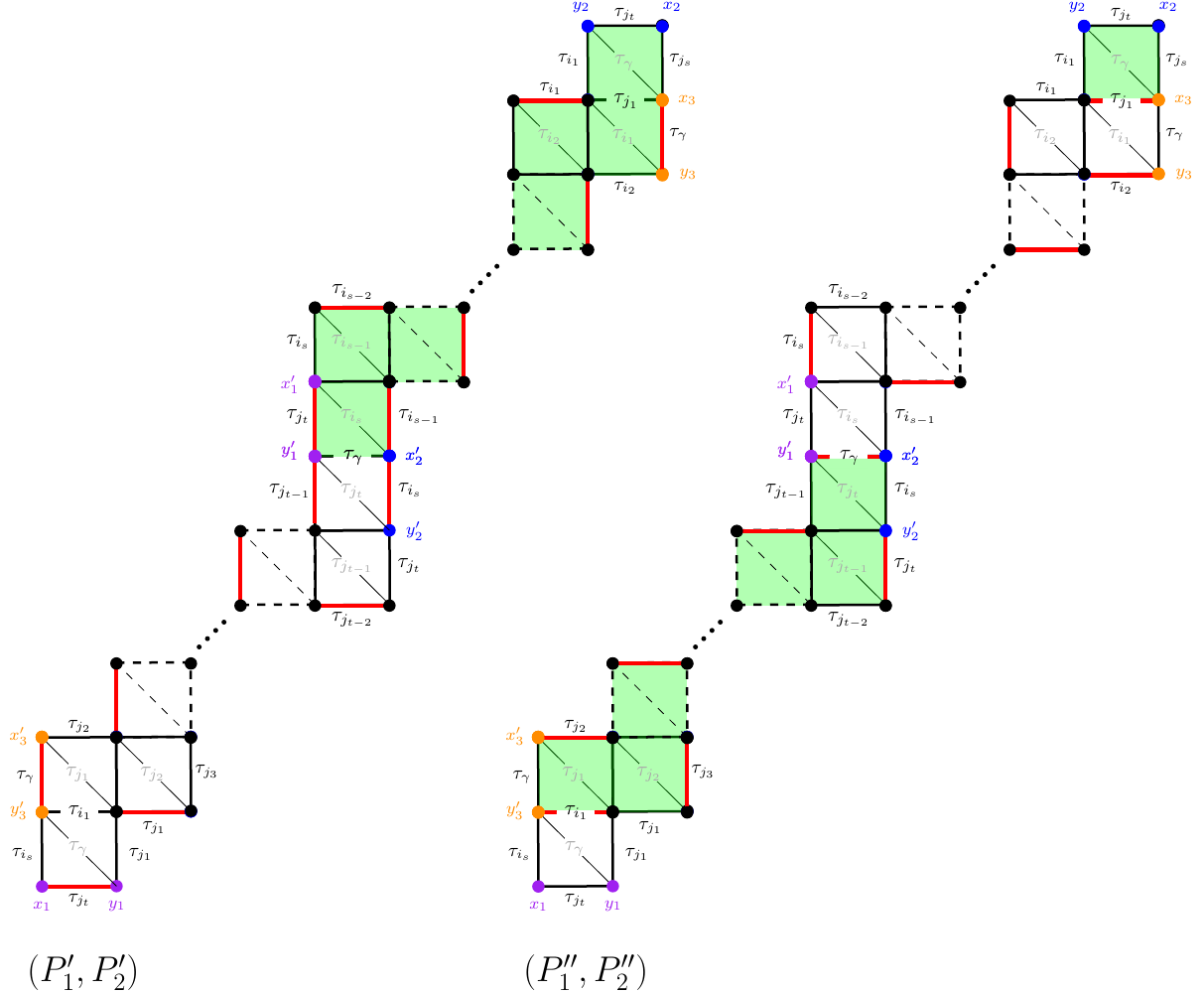}
\caption{As Proposition \ref{DoubleUnderlyingExpansion} alludes to (and the Proof of Theorem \ref{loopexpansion} when $\gamma^{\circ} \in T$ explains), in essence, the ``product'' of all perfect matchings of $G_{\ell_p, T^{\circ}}$ and $G_{\ell_q, T^{\circ}}$ overcount the expansion of $x_{\gamma^{(pq)}}$ by two terms, but also undercount it by missing two terms. This figure shows the two pairs of perfect matchings $(P_1',P_2')$ and $(P_1'',P_2'')$ that cause this ``overcounting''. One should further note that, as expected, these two pairs do not give rise to good matchings of the loop graph $G_{\gamma^{(pq)},T}$. }
\label{doubleloopgraphminus}
\end{center}
\end{figure}

%\frac{x_{\gamma^{(p)}}x_{\gamma^{(q)}}y_{\gamma}}{\text{cross}(T,\alpha) over counts
%

\begin{figure}[H]
\begin{center}
\includegraphics[width=17cm]{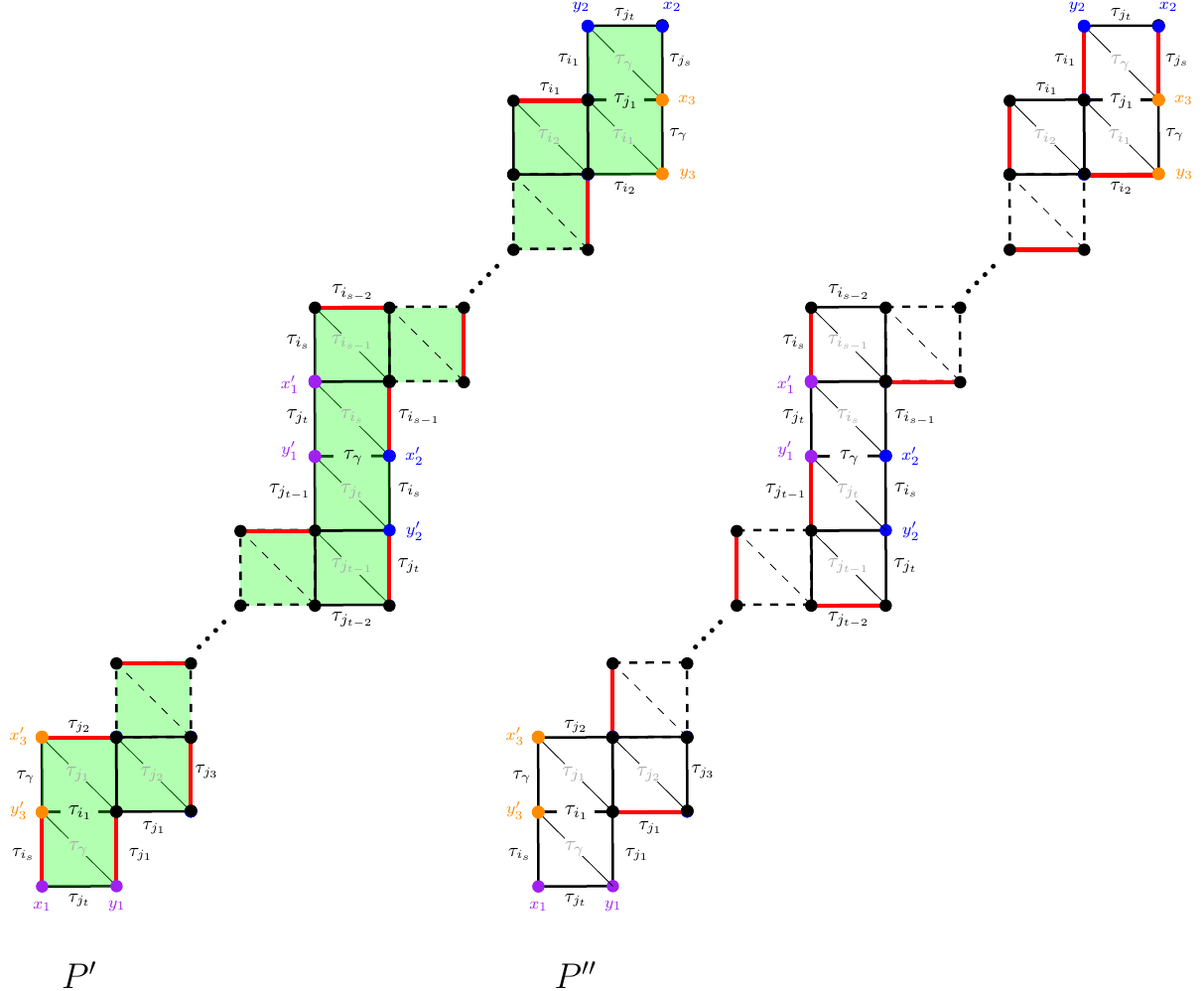}
\caption{As Proposition \ref{DoubleUnderlyingExpansion} alludes to (and the Proof of Theorem \ref{loopexpansion} when $\gamma^{\circ} \in T$ explains), in essence, the ``product'' of all perfect matchings of $G_{\ell_p, T^{\circ}}$ and $G_{\ell_q, T^{\circ}}$ overcount the expansion of $x_{\gamma^{(pq)}}$ by two terms, but also undercount it by missing two terms. This figure shows the two good matchings $P'$ and $P''$ that cause this ``undercounting''. Indeed, note that $P'$ and $P''$ do not arise from the ``product'' of perfect matchings of $G_{\ell_p, T^{\circ}}$ and $G_{\ell_q, T^{\circ}}$ as they possess right cuts on the first and last tiles of $G_{\gamma^{(pq)},T}$, respectively.}
\label{doubleloopgraphplus}
\end{center}
\end{figure}

\newpage

\section{The lattice structure of good matchings of loop graphs}
\label{latticesection}
In this section we discuss some properties of surface loop graphs. It builds on the work of Musiker, Schiffler and Williams established for snake graphs in [Section 5, \cite{musiker2013bases}]. We shall assume throughout that the first tile of any loop graph has relative orientation $1$.

\begin{defn}

Let $P$ be a good matching of a loop graph $\mathcal{G}^{\bowtie} = (G_1 \ldots, G_d)^{\bowtie}$. We say a good matching $P'$ of $\mathcal{G}^{\bowtie}$ is obtained from $P$ by a \textit{\textbf{positive twist}} at tile $G_j$ if one of the following holds:

\begin{itemize}

\item $P \setminus \{N(G_j),S(G_j)\} = P' \setminus \{W(G_j),E(G_j)\}$ and $j$ is odd, or

\item $P \setminus \{W(G_j),E(G_j)\} = P' \setminus \{N(G_j),S(G_j)\}$ and $j$ is even.

\end{itemize}

\end{defn}

\begin{defn}
\label{lattice of good matchings}
Let $\mathcal{G}^{\bowtie}$ be a surface loop graph. We define the \textbf{\textit{lattice of good matchings}} of $\mathcal{G}^{\bowtie}$ to be the lattice $L(\mathcal{G}^{\bowtie})$ whose underlying set is the collection of all good matchings of $G^{\bowtie}$. The lattice structure $(L(\mathcal{G}^{\bowtie}), \vee, \wedge)$ is given as follows:

\begin{itemize}

\item $(L(G^{\bowtie}),\leq)$ is viewed as a poset by demanding that for each $P_1, P_2 \in L(G^{\bowtie})$ we have $P_1 \leq P_2$ \textit{if and only if} $P_1 = P_2$, or $P_2$ is obtained from $P_1$ by a sequence of positive twists.

\item For any $P_1, P_2 \in L(G^{\bowtie})$ the join $P_1 \vee P_2 \in L(G^{\bowtie})$ is the (unique) good matching such that $$\text{for any $P_3 \in L(G^{\bowtie})$, if $P_1 \leq P_3$ and $P_2 \leq P_3$, then $P_1 \vee P_2 \leq P_3$.}$$

\item Likewise, for any $P_1, P_2 \in L(G^{\bowtie})$ the meet $P_1 \wedge P_2 \in L(G^{\bowtie})$ is the (unique) good matching such that $$\text{for any $P_3 \in L(G^{\bowtie})$, if $P_3 \leq P_1$ and $P_3 \leq P_2$, then $P_3 \leq P_1 \wedge P_2$.}$$

\end{itemize}

\begin{rmk}

A priori, it is not obvious the lattice structure defined above even makes sense. However, Theorem \ref{loop order lattice} below ensures Definition \ref{lattice of good matchings} is well defined, and further shows the lattice is \textit{distributive} -- see \cite{birkhoff1937rings} or [Section 3.4, \cite{stanley2011enumerative}].

\end{rmk}

\end{defn}

\begin{defn}
\label{snake quiver}
Let $\mathcal{G} = (G_1, \ldots, G_d)$ be a snake graph. We define $Q_{\mathcal{G}}$ to be the quiver with vertices $\{1, \ldots, d\}$ and whose arrows are determined by the following rules:

\begin{itemize}

\item there is an arrow $i \rightarrow i+1$ in $Q_{\mathcal{G}}$ \textit{if and only if} $i$ is odd and $G_{i+1}$ is on the right of $G_i$, or $i$ is even and $G_{i+1}$ is on top $G_i$,

\item there is an arrow $i \rightarrow i-1$ in $Q_{\mathcal{G}}$ \textit{if and only if} $i$ is odd and $G_{i}$ is on the right of $G_{i-1}$, or $i$ is even and $G_{i}$ is on top of $G_{i-1}$.

\end{itemize}

$Q_{\mathcal{G}}$ induces a poset structure on $\{1,\ldots, d\}$ by setting $i \leq j$ \textit{if and only if} $i = j$ or there is a sequence of arrows in $Q_{\mathcal{G}}$ oriented from $i$ to $j$. Namely, $Q_{\mathcal{G}}$ is the Hasse diagram of this poset.

\end{defn}

\begin{defn}

We say $I \subseteq Q$ is an \textit{\textbf{order ideal}} of a poset $(Q, \leq)$ if whenever $x \in I$ and $y \in Q$ satisfy $y \leq x$, then $y \in I$.

\end{defn}

\begin{defn}

Let $P$ be a good matching of a loop graph $\mathcal{G}^{\bowtie} = (G_1, \ldots, G_d)^{\bowtie}$. We define the \textit{\textbf{height}} to be the set $h(P) \subseteq \{1,\ldots, d\}$ where $ i \in h(P)$ \textit{if and only if} the diagonal of $G_i$ is positive with respect to $P$.

\end{defn}

\begin{thm}[Theorem 5.4, \cite{musiker2013bases}]
\label{order lattice}
Let $\mathcal{G}$ be a snake graph. Then $L(\mathcal{G})$ is isomorphic to the (distributive) lattice of order ideals $\mathcal{J}(Q_{\mathcal{G}})$ of the poset $Q_{\mathcal{G}}$. Specifically, this isomorphism is given by sending a perfect matching $P$ to its height $h(P)$.

\end{thm}

\begin{defn}
\label{WrittenHasseQuiverDefinition}

Let $\mathcal{G}^{\bowtie} = (G_1, \ldots, G_d)^{\bowtie}$ be a loop graph with underlying snake graph $\mathcal{G} = (G_1, \ldots, G_d)$. We define the quiver $Q_{\mathcal{G}^{\bowtie}}$ of $\mathcal{G}^{\bowtie}$ to be the quiver $Q_{\mathcal{G}}$ with an additional arrow for each loop or band of $\mathcal{G}^{\bowtie}$. Specifically, for $k_1,k_2 \in \{1, \ldots, d\}$ where $k_1 < k_2$ this arrow is determined by the following rule (see Figures \ref{HasseQuiverDefinition} and \ref{HasseQuivers}):

\begin{itemize}

\item if there is a loop with respect to $G_{k_1}$ and $G_{k_2}$ and $k_1$ is odd, then the arrow $k_1 \rightarrow k_2$ (resp. $k_2 \rightarrow k_1$) is in $Q_{\mathcal{G}^{\bowtie}}$ \textit{if and only if} $k_2$ is odd (resp. even) and $S(G_{k_2})$ is the cut edge, or $k_2$ is even (resp. odd) and $W(G_{k_2})$ is the cut edge. The definition is analogous if $k_1$ is even -- one should just interchange South and West.

\item if there is a loop with respect to $G_{k_2}$ and $G_{k_1}$ and $k_1$ is odd, then the arrow $k_2 \rightarrow k_1$ (resp. $k_1 \rightarrow k_2$) is in $Q_{\mathcal{G}^{\bowtie}}$ \textit{if and only if} $k_2$ is odd (resp. even) and $N(G_{k_1})$ is the cut edge, or $k_2$ is even (resp. odd) and $E(G_{k_1})$ is the cut edge. The definition is analogous if $k_1$ is even -- one should just interchange North and East.

\item if there is a band with respect to $G_{k_1}$ and $G_{k_2}$ and $k_1$ is odd, then the arrow $k_1 \rightarrow k_2$ (resp. $k_2 \rightarrow k_1$) is in $Q_{\mathcal{G}^{\bowtie}}$ \textit{if and only if} $k_2$ is odd (resp. even) and $N(G_{k_2})$ is the cut edge, or $k_2$ is even (resp. odd) and $E(G_{k_2})$ is the cut edge. The definition is analogous if $k_1$ is even -- one should just interchange North and East.

\end{itemize}

\end{defn}

\begin{figure}[H]
\begin{center}
\includegraphics[width=16cm]{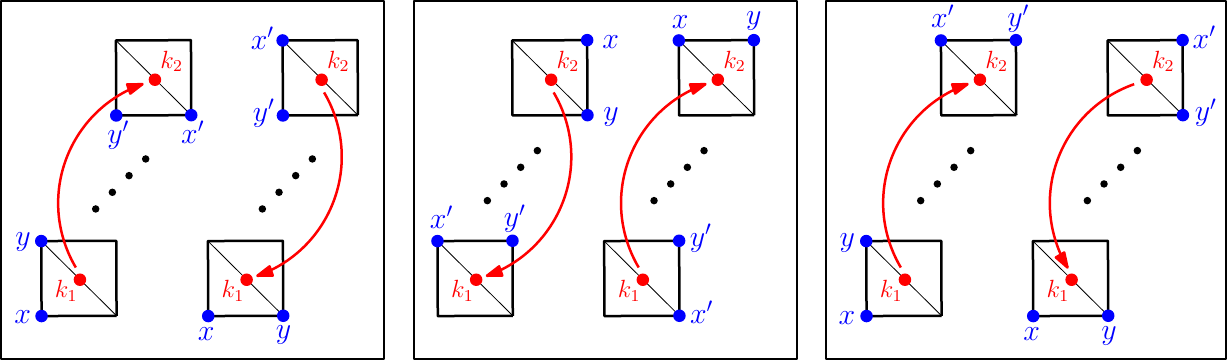}
\caption{In the case when $k_1$ and $k_2$ are both odd, we list the various types of additional arrow described in Definition \ref{WrittenHasseQuiverDefinition} with regards to a loop with respect to $G_{k_1}$ and $G_{k_2}$; a loop with respect to $G_{k_2}$ and $G_{k_1}$; and a band with respect to $G_{k_1}$ and $G_{k_2}$.}
\label{HasseQuiverDefinition}
\end{center}
\end{figure}

\begin{lem}
\label{posetlemma}
Let $\mathcal{G}^{\bowtie} = (G_1, \ldots, G_d)^{\bowtie}$ be a surface loop graph. Then $Q_{\mathcal{G}^{\bowtie}}$ defines a well-defined poset structure on $\{1,\ldots, d\}$ by setting $i \leq j$ \textit{if and only if} $i = j$ or there is a sequence of arrows in $Q_{\mathcal{G}^{\bowtie}}$ oriented from $i$ to $j$.

\end{lem}

\begin{proof}

Suppose there is a loop with respect to $G_1$ and $G_k$ where $k \in \{3,\ldots, d\}$. Since $\mathcal{G}^{\bowtie}$ is a surface loop graph then $(G_{k-2},G_{k-1},G_k)$ and $(G_{2},G_{1},G_k)$ are straight subgraphs of $\mathcal{G}^{\bowtie}$ (due to the corresponding arcs $(\gamma_{k-2},\gamma_{k-1},\gamma_{k})$ and $(\gamma_{2},\gamma_{1},\gamma_{k})$ forming an ``s'' or ``z'' shape on the surface, rather than a fan). An analogous statement holds if there is a loop with respect to $G_d$ and $G_k$ where $k \in \{1,\ldots, d-2\}$, or a band with respect to $G_2$ and $G_{d-1}$. Consequently, the poset structure is well defined.

\end{proof}

\begin{figure}[H]
\begin{center}
\includegraphics[width=16.5cm]{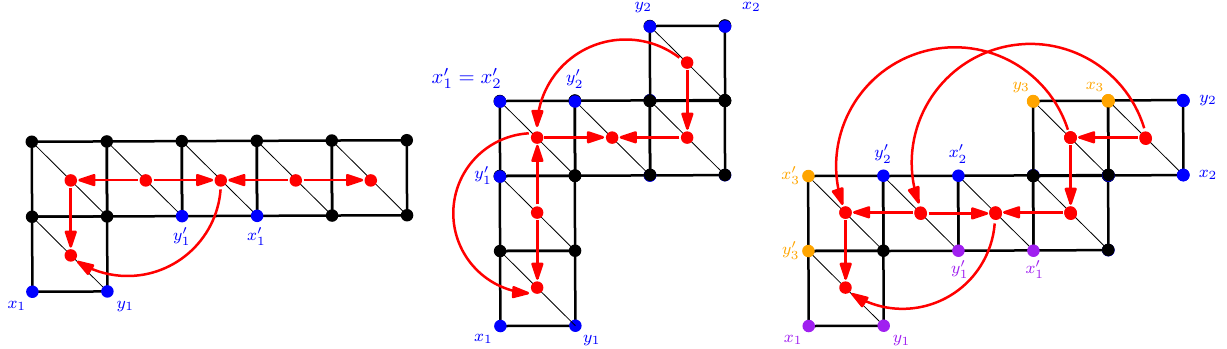}
\caption{The associated quivers, shown in red, of the loop graphs found in Figures \ref{single}, \ref{double} and \ref{DoubleUnderlying}.}
\label{HasseQuivers}
\end{center}
\end{figure}

We are now ready to state the main theorem of this section.

\begin{thm}
\label{loop order lattice}

Let $\mathcal{G}^{\bowtie}$ be a surface loop graph. Then $L(\mathcal{G}^{\bowtie})$ is isomorphic to the lattice of order ideals $\mathcal{J}(Q_{\mathcal{G}^{\bowtie}})$ of the poset $Q_{\mathcal{G}^{\bowtie}}$. Specifically, this isomorphism is given by sending a good matching $P$ to its height $h(P)$.

\end{thm}

\begin{proof}

Let us consider the case when $\mathcal{G}^{\bowtie}$ arises from a singly notched arc. Moreover, let $\mathcal{G} = (G_1,\ldots, G_k, \ldots, G_d)$ be a snake graph which gives rise to $\mathcal{G}^{\bowtie}$ by creating a loop with respect to $G_1$ and $G_k$. Without loss of generality, we may assume that $G_2$ is on the right of the tile $G_1$, and that $(G_{k-1},G_k,G_{k+1})$ is a zig-zag (we allow the possibility that $G_{k+1} = \emptyset$) -- see Figure \ref{latticesetup}. \newline \indent 

If $\overline{P}$ is a perfect matching of $\mathcal{G}$ that induces a negative diagonal on tile $G_k$ then the edge matching $x'$ and $y'$ is in $\overline{P}$. So $\overline{P}$ descends to a good matching on $\mathcal{G}^{\bowtie}$. \newline \indent Conversely, If $\overline{P}$ is a perfect matching of $\mathcal{G}$ that induces a positive diagonal on tile $G_k$ then: $N(G_k) \in \overline{P}$ or $E(G_k) \in \overline{P}$, respective of whether $k$ is even or odd. So $\overline{P}$ descends to a good matching of $\mathcal{G}^{\bowtie}$ \textit{if and only if} the edge matching $x$ and $y$ is in $\overline{P}$. Consequently, the diagonal of $G_1$ must also be positive. \newline \indent

Therefore, by Theorem \ref{order lattice}, the collection of good matchings of $\mathcal{G}^{\bowtie}$ may be identified with the order ideals of $Q_{\mathcal{G}}$ which include the vertex $1$ whenever they contain the vertex $k$. This subcollection of order ideals is precisely the order ideals of $Q_{\mathcal{G}^{\bowtie}}$, which completes the proof when $\mathcal{G}^{\bowtie}$ arises from a singly-notched arc. The doubly-notched case follows in exactly the same way.

\end{proof}

\begin{figure}[H]
\begin{center}
\includegraphics[width=14cm]{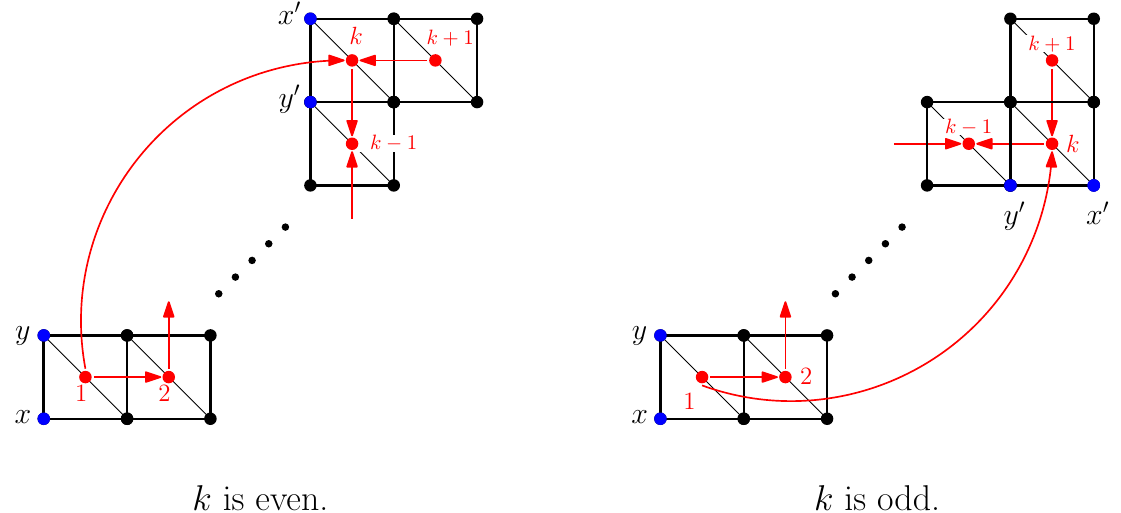}
\caption{The setup of the loop graph $\mathcal{G}^{\bowtie}$ discussed in the proof of Theorem \ref{loop order lattice}. The red quiver indicates (part of) the associated quiver $Q_{\mathcal{G}^{\bowtie}}$.}
\label{latticesetup}
\end{center}
\end{figure}

\begin{cor}
\label{minmaxmatchings}
Let $\mathcal{G}^{\bowtie}$ be a surface loop graph. Then there exist (unique) good matchings $P_-$ and $P_+$ of $\mathcal{G}^{\bowtie}$ for which the diagonals of $\mathcal{G}^{\bowtie}$ are all negative and all positive, respectively. We call these the \textbf{\textit{minimal}} and \textbf{\textit{maximal}} good matchings of $\mathcal{G}^{\bowtie}$.

\end{cor}

\begin{proof}

Note that $\emptyset$ and $\{1,\ldots, d\}$ are order ideals of $Q_{\mathcal{G}^{\bowtie}} = (G_1, \ldots, G_d)^{\bowtie}$. By Theorem \ref{loop order lattice} there exist good matchings $P_-$ and $P_+$ of $\mathcal{G}^{\bowtie}$ such that $h(P_-) = \emptyset$ and $h(P_+) = \{1,\ldots, d\}$.

\end{proof}

\begin{rmk}
\label{repremark}

Let $\mathcal{G}^{\bowtie}$ be a loop graph. One can associate a quiver representation $M(I)$ of $Q_{\mathcal{G}^{\bowtie}}$ to each order ideal $I$ of $Q_{\mathcal{G}^{\bowtie}}$. Specifically, $$M(i) = \mathbb{C} \hspace{4mm} \textit{if and only if} \hspace{4mm} i \notin I$$ and for each arrow $\alpha: i \rightarrow j$ in $Q_{\mathcal{G}^{\bowtie}}$ we have $$M(\alpha) = Id, \hspace{2mm} \text{if} \hspace{2mm} i,j \notin I \hspace{6mm} \text{and} \hspace{6mm} M(\alpha) = 0, \hspace{3mm} \text{otherwise}.$$ Adopting the terminology used in \cite{canakci2021lattice}, we see from Theorem \ref{loop order lattice} that $L(\mathcal{G}^{\bowtie})$ is isomorphic to the \textit{canonical} submodule lattice of $M(\emptyset)$, with respect to $Q_{\mathcal{G}^{\bowtie}}$, for any surface loop graph $\mathcal{G}^{\bowtie}$. Note that our situation is dual to that presented in \cite{canakci2021lattice}, so the ordering on the canonical submodule lattice is defined by $M(I) \leq M(J)$ \textit{if and only if} $M(J)$ is a submodule of $M(I)$.
\end{rmk}

In Remark \ref{mainremark} we spoke about the so called poset version of equation (\ref{loopexpansionformula}). We conclude the paper by describing this alternative version, as found in \cite{huang2021new}, \cite{ouguz2025cluster},\cite{pilaud2023posets},\cite{weng2023f}, and demonstrate that it naturally follows from equation (\ref{loopexpansionformula}) and the results of this section.

\begin{defn}
\label{posetweight}
Let $T^{\circ} = \{\tau_1, \ldots, \tau_n\}$ be an ideal triangulation with associated tagged triangulation $T := \iota(T)$. Moreover, given a directed tagged arc $\gamma$, consider the poset $P_{G_{\gamma,T}}$ (recall that each vertex of this poset is labelled by an arc $\tau \in T^{\circ}$, and that multiple vertices may share the same label). \newline
Then for each order ideal $I$ of $Q_{G_{\gamma,T}}$, following the conventions outlined in sections \ref{preliminariessection} and \ref{loopgraphsection}, we define the \textit{\textbf{poset weight monomial}} $w(I)$ as follows:

\[
w(I) := \displaystyle \prod_{\tau \in I} w_{\tau} \hspace{10mm} \text{where} \hspace{10mm} w_{\tau} := 
\begin{cases} 
      \frac{\hat{y}_{\tau}}{\hat{y}_{{\tau}^{(p)}}} & \text{if $\tau_{j}$ is a radius, }\\ 
          \hat{y}_{\tau} & \text{otherwise.}
\end{cases}
\]

where $\tau^{(p)} \in T$ is obtained from the plain (radius) arc $\tau \in T^{\circ}$ by adding a notch at the associated puncture $p$.

\end{defn}

\begin{cor}
\label{maincor}
Let $(S,M)$ be a bordered surface, and let $T^{\circ}$ be an ideal triangulation with corresponding tagged triangulation $T= \iota(T^\circ)$. We define $\mathcal{A}$ to be the associated cluster algebra with principal coefficients with respect to $\Sigma_T = (\mathbf{x}_T,\mathbf{y}_T,B_T)$. \newline Let $\gamma$ be a tagged arc of $(S,M)$ (as in Theorem \ref{loopexpansion}, if $\gamma$ is doubly notched then we also assume $(S,M)$ is not twice-punctured and closed), and let $\mathbf{g(\gamma},T) = (g_1,g_2, \ldots, g_n)$ be the g-vector of $x_{\gamma}$ with respect to $T$. Then the Laurent expansion of $x_{\gamma} \in \mathcal{A}$ with respect to $\Sigma_T$ is given by:

\begin{equation}
\label{posetexpansionformula}
x_{\gamma} = {\mathbf{x}_T}^{\mathbf{g(\alpha},T)}\sum_{I} w(I),
\end{equation}

\noindent where the sum is taken over all order ideals $I$ of the poset $Q_{G_{\gamma,T}}$ and ${\mathbf{x}_T}^{\mathbf{g(\gamma},T)} := x_1^{g_1} x_2^{g_2} \ldots x_n^{g_n}$.

\end{cor}

\begin{proof}

Note that, by Theorem \ref{loopexpansion}, Corollary \ref{minmaxmatchings}, and the definition of $g$-vector, we have $${\mathbf{x}_T}^{\mathbf{g(\alpha},T)} =  \frac{x(P_{-})}{\text{cross}(\gamma, T)}$$ where $P_{-}$ is the minimal good matching of the loop graph $G_{\gamma, T}$.

Furthermore, following \cite{fomin2007cluster}, recall that the so called \textit{\textbf{F-polynomial}}, $F_{\gamma,T}(y_1,\ldots, y_n) \in \mathbb{Z}[y_1,\ldots, y_n]$, is defined to be cluster variable $x_{\gamma,T}$ specialised at $x_1 = \ldots = x_n = 1$. Therefore, by Theorem \ref{loopexpansion}, $F_{\gamma,T}(y_1,\ldots,y_n) = \displaystyle \sum_{P} y(P)$ where the sum is over all good matchings $P$ of the loop graph $G_{\gamma,T}$.

Finally, coupling this with Theorem \ref{loop order lattice}, Definitions \ref{goodmatchingcoefficientmonomial} and \ref{posetweight}, and the refined separation of addition formulae of Fomin and Zelevinsky [Corollary 6.3, \cite{fomin2007cluster}] we thus obtain $$\displaystyle x_{\gamma,T} \ = \ {\mathbf{x}_T}^{\mathbf{g(\alpha},T)}F(\hat{y}_1,\ldots,\hat{y}_n)  \ = \ {\mathbf{x}_T}^{\mathbf{g(\alpha},T)}{\left(\sum_{P} y(P)\right)}_{\vert_{y_i \rightarrow \hat{y}_i}} =  \ \ {\mathbf{x}_T}^{\mathbf{g(\alpha},T)}\sum_{I} w(I).$$

\noindent where the final sum is taken over all order ideals $I$ of the poset $Q_{G_{\gamma,T}}$.

\end{proof}

\begin{rmk}

The above Corollary \ref{maincor} continues to be true even when $(S,M)$ is a twice-punctured closed bordered surface and $\gamma$ is doubly notched. Indeed, as explained in Remark \ref{mainremark}, our Main Theorem \ref{loopexpansion} also holds in this setting, so the same proof of the corollary applies.

\end{rmk}

\newpage

\bibliography{bases}
\bibliographystyle{plain}

\Addresses

\end{document}